\documentclass[12pt]{article}
\usepackage[english]{babel}
\usepackage{geometry}
\usepackage[
backend=biber,
style=alphabetic,
citestyle=numeric,
bibstyle=numeric,
sorting=nyt
]{biblatex}
\usepackage{amsthm}
\usepackage{hyperref}
\usepackage{graphicx}
\usepackage{newtxmath}
\usepackage{subcaption}
\usepackage{float}
\usepackage[dvipsnames]{xcolor}
\usepackage[linesnumbered, boxruled, vlined]{algorithm2e}
\usepackage{multirow}
\usepackage{cleveref}

\crefname{equation}{}{}
\Crefname{equation}{}{}

\newtheorem{theorem}{Theorem}[section]
\newtheorem*{theorem*}{Theorem}
\newtheorem{proposition}{Proposition}[section]
\newtheorem*{proposition*}{Proposition}
\newtheorem{corollary}{Corollary}[section]
\newtheorem*{corollary*}{Corollary}
\newtheorem{lemma}{Lemma}[section]
\newtheorem*{lemma*}{Lemma}
\newtheorem{definition}{Definition}[section]
\newtheorem*{definition*}{Definition}

\newcommand{\R}{\mathbb{R}}
\newcommand{\K}{\mathcal{K}}
\newcommand{\Pk}{\mathcal{P}_{\K_k}}
\newcommand{\Pks}{\mathcal{P}_{\K_k}^\Omega}
\newcommand{\kmax}{k_{\text{max}}}
\newcommand{\comment}[1]{}

\newcommand{\undemi}{\frac{1}{2}}
\newcommand{\Span}[1]{\mathrm{Span} \langle #1 \rangle}

\SetAlgoSkip{bigskip}
\SetCommentSty{mycommfont}
\SetKwInput{KwInput}{Input}                
\SetKwInput{KwOutput}{Output}

\DeclareMathOperator*{\argmin}{arg\,min}

\title{Randomized Orthogonal Projection Methods for Krylov Subspace Solvers}
\author{Edouard Timsit\thanks{Sorbonne Université, Inria, CNRS, Université de Paris, Laboratoire Jacques-Louis Lions, Paris, France}, Laura Grigori\footnotemark[1], Oleg Balabanov\thanks{Part of this work was conducted while the author was at Sorbonne Université, Inria, CNRS, Université de Paris, Laboratoire Jacques-Louis Lions, Paris, France}}

\date{March 10th, 2023}

\usepackage[utf8]{inputenc}

\DeclareUnicodeCharacter{1EF3}{\`y}

\begin{document}

\maketitle

\comment{\tableofcontents}

\textbf{Abstract.} 
Randomized orthogonal projection methods  (ROPMs) can be used to speed up the computation of Krylov subspace methods in various contexts. Through a theoretical and numerical investigation, we establish that these methods produce quasi-optimal approximations over the Krylov subspace. Our numerical experiments outline the convergence of ROPMs for all matrices in our test set, with occasional spikes, but overall with a convergence rate similar to that of standard OPMs.

\section{Introduction}
Solving systems of linear equations of large dimension arise in a variety of applications. Given such a system of linear equations 
\begin{align}\label{baseequation}
A x = b,
\end{align}
where $A \in \mathbb{R}^{n \times n}$, $b \in \mathbb{R}^n$, and given a \textit{first guess} $x_0 \in \mathbb{R}^n$, Krylov solvers build the subspace 
\begin{align}\label{krylovsubspace}
\mathcal{K}_k(A,r_0) := \text{Span} \{ r_0, A r_0 , \, \cdots \, , A^{k-1} r_0 \},
\end{align}
where $k \ll n$ and $r_0 = b - A x_0$, and search for an approximation $\breve{x}_k \in \K_k(A,r_0)$ of the solution of~\cref{baseequation}. For example, this approximation can be obtained by minimizing the residual norm $\|b - A \breve{x}_k\|$ or the $A$-norm of the error $\|x - \breve{x}_k\|_A$ when $A$ symmetric positive definite. In the Krylov subspace mathematical setting, this minimization is equivalently formulated in terms of orthogonality between the residual and some $k$-dimensional subspace. This orthogonality condition is called the Petrov-Galerkin condition. Since the canonical basis of the Krylov subspace is expected to be ill-conditioned, using a well conditioned basis of this subspace lies at the heart of Krylov methods. For this reason, Krylov solvers usually proceed with a subsequent orthogonalization of the basis of $\K_k(A,r_0)$.

However most orthogonalization processes of $k$ vectors of $\R^n$ have a $\mathcal{O}(n k^2)$ cost in terms of arithmetic operations. This cost is likely to dominate and define the overall cost of the Krylov solver. It can be mitigated through the use of a dimension reduction technique called \textit{random sketching}. Given a $k$-dimensional vector subspace $\mathcal{V}_k$ of $\R^n$, $k \ll n$, this technique embeds $\mathcal{V}_k$ into $\R^\ell$, $k < \ell \ll n$, through a linear mapping $\Omega \in \R^{\ell \times n}$. The number of rows of $\Omega$ is called the \textit{sampling size}. One of the features of this linear mapping is that it preserves the inner products between pairs of vectors of $\mathcal{V}_k$,
up to a certain tolerance $\epsilon$, with high probability. The probability of success, as well as the $\epsilon$ we wish to attain, prescribe the sampling size $\ell$. The inner products of two vectors $x, y \in \mathcal{V}_k$ are then performed on their low-dimensional representations $\Omega x, \Omega y$, called their sketches, at a lower computational cost. The result is an approximation of the inner products between $x$ and $y$. The orthonormalization of $k$ vectors of $\R^n$ vectors is then replaced by an orthonormalization of their $k$ sketches in $\R^\ell$ (\textit{sketched orthonormalization}) for a lesser cost. It is then natural to investigate the use of sketched orthonormalization instead of the traditional orthogonalization process inside a Krylov solver.

In the context of Krylov methods, the random sketching technique was first used in~\cite{rgs}, where the authors developed a randomized version of the Gram-Schmidt (RGS) algorithm and applied it to the Arnoldi iteration and GMRES methods. It has been shown that the RGS algorithm is more efficient than traditional Gram-Schmidt algorithms, such as the modified and double Gram-Schmidt processes, in terms of arithmetic and communication cost, while being as stable as these methods. The methodology from~\cite{rgs} was later extended to block versions of RGS (BRGS) and the corresponding randomized Arnoldi and GMRES algorithms~\cite{BRGS}. In addition, the authors in~\cite{BRGS} suggested applying their methodology to other well-known Krylov methods, including the Rayleigh-Ritz approximation for eigenproblems and the FOM for linear systems (as described in \cite[Remark4.1]{BRGS}), by using a randomized (or sketched) version of the Petrov-Galerkin condition. The technical details of this approach are presented in the next section. This idea provides the foundation for the present work. It is worth noting that the concept of the sketched Petrov-Galerkin condition was first introduced in~\cite{Galerkin}, where it was applied to compute approximate solutions of linear parametric systems in a low-dimensional subspace. The main difference between present work and that in~\cite{Galerkin} is that the latter considers general subspaces constructed using methods such as greedy algorithms, whereas the present work focuses on Krylov subspaces, which appear to be better suited for random sketching and its analysis. 

Recently, randomization has been considered in~\cite{alice,guttel} to speed up the computation of the matrix function $f(A)b$ (where linear systems are a special case with $f(A) = A^{-1}$). While~\cite{alice} relies on partial orthogonalization methods and RGS,~\cite{guttel}~uses the sketched Petrov-Galerkin condition (or RFOM condition). Compared to~\cite{guttel}, the present work provides a more complete characterization of the accuracy of RFOM for linear systems in terms of both theoretical analysis and numerical experiments. It is important to note that although our work is concerned with solution of linear systems, the developed here results have the potential to characterize RFOM applied not only to the approximation of $A^{-1}b$, but also to general matrix functions due to the close connection between RFOM for matrix functions and the approximation of a series of shifted linear systems with sketched Petrov-Galerkin projections onto the Krylov subspace, as discussed in~\cite{guttel}. The sketched Petrov-Galerkin condition has also been used in~\cite{fastrandalg}, but only for the approximate solution of eigenvalue problems. Linear systems in~\cite{fastrandalg} were solved using a sketched minimal-residual condition similar to that in~\cite{rgs}. It is important to note that the methods in~\cite{guttel,fastrandalg,rgs,BRGS} are all mathematically equivalent in the sense that they either minimize the sketched residual or impose the sketched Petrov-Galerkin condition. The main difference between~\cite{guttel,fastrandalg} and~\cite{rgs,BRGS}  effectively lies in the numerical construction (possibly implicitly) of the sketched orthogonal Krylov basis  that satisfies the Arnoldi identity. The algorithms from~\cite{rgs,BRGS} are based on a monolithic construction using a randomized Arnoldi algorithm, while the methods in~\cite{guttel,fastrandalg} use a fast (but less stable) deterministic algorithm, such as the truncated Arnoldi, augmented with the randomized Cholesky QR (or more precisely the reduced randomized Cholesky QR depicted in~\cite{rcholqr}).

Through a theoretical and numerical investigation, this paper presents a comprehensive study of randomized orthogonal projection methods (ROPMs), a class of algorithms that includes RFOM, establishing $x_k$ as an overall tight quasi-optimal approximation of the solution of \cref{baseequation}. We introduce bounds characterizing the estimate produced by ROPMs as quasi-optimal minimizer of the $A$-norm of the error (energy norm) when $A$ is symmetric positive definite (SPD). Numerical investigation outlines the convergence of ROPMs for all matrices in our test set, with a convergence rate similar to that of standard OPMs. We also highlight occasional spikes of the $A$-norm of the error $\|x - x_k\|_A$, that can be several orders of magnitude higher than that of the error of standard OPMs $\|x - \breve{x}_k\|_A$. 
The estimate $x_k$, obtained by enforcing the sketched Petrov-Galerkin condition, can be straightforwardly characterized as a quasi-optimal minimizer of the $A$-norm of the error when the embedding $\Omega$ is successful, 
\begin{align}\label{quasi1intro}
	 \|x - x_k \|_A \leq \frac{1 + \epsilon \, \mathrm{cond}(A)^\frac{1}{2}}{1 - \epsilon \, \mathrm{cond}(A)^\frac{1}{2}} \| x - \breve{x}_k \|_A.
\end{align}
However, when the condition number of $A$ is large, this bound is not tight, and requires a large sampling size to be well-defined. Indeed, the inequality $\epsilon < \mathrm{cond}(A)^{-\undemi}$ needs to be satisfied for the denominator $1 - \epsilon \mathrm{cond}(A)$ to be positive. Furthermore, it does not accurately describe the subtle behavior of the solver, which in our numerical experiments follows tightly the monotonous error of the deterministic method but can feature occasional spikes of several orders of magnitude. Given this observation, we derive a second bound that accurately describes the behavior of the solver observed in our numerical experiments,
\begin{align}\label{quasi2intro}
\| x - x_k\|_A \leq \left (1+\alpha^2_k \beta^2_k \right)^{\frac{1}{2}} \| x - \breve{x}_k\|_A, 
\end{align}
where the factor involving $\alpha_k$ and $\beta_k$ is expected to capture the spikes, and be of the order of $1$ otherwise. This might allow to foresee the simple conditions under which the method will work, and potentially avoid those spikes ($\alpha_k, \beta_k$ are defined in~\Cref{section:quasi}).

In addition, we propose an a posteriori bound linking the residual produced by ROPM to that produced by the deterministic equivalent. We also discuss the usage of randomization for short recurrence algorithms, and the difficulty of benefiting from the symmetry of $A$ with randomization, making the prospect of randomizing MinRES and CG more complex than the randomization of GMRES developed in~\cite{rgs}.  

We finally propose a straightforward approximation through random sketching of CG algorithm, with a posteriori error estimation based on the methodology developed in~\cite{whycg}. We identify two types of spectral profiles for which this algorithm can be used, saving one communication per iteration in a distributed environment. However in general this algorithm cannot be expected to be stable.  

The paper is organized as follows. In~\Cref{section:preliminaries} we introduce in more details the notions that are used in our work. In~\Cref{section:quasi} we prove and discuss the characterizations of $x_k$ from \cref{quasi1intro,quasi2intro} mentioned in this introduction. In~\Cref{section:rarnoldi} we prove an indirect bound linking residuals produced by both randomized and deterministic projection methods. In~\Cref{section:symmetry}, we discuss the difficulties caused by random sketching as to sketched orthonormalizing a basis of the Krylov subspace through short-recurrence. In~\Cref{section:arCG} we introduce the approximation of CG algorithm through randomization, that can be useful in some applications, and derive a posteriori error bound for the $A$-norm of the error. Finally, in~\Cref{section:results}, we report experiments that showcase our theoretical results. 

\section{Preliminaries} \label{section:preliminaries}
\subsection{Notations} \label{section:preliminaries_notations}

We use standard linear algebra notation, such as in~\cite{saad}, with varying accents depending on which algorithm is discussed. Because they are at the center of the discussion, the byproducts of randomized methods are given the simplest notation without accents. For instance, in iterative randomized algorithms the $k$-th estimate of $x$ and the $k$-th residual are denoted by $x_k$ and $r_k$, respectively. On the other hand, byproducts of deterministic algorithms are singled out by a \textit{breve} accent. For instance, the $k$-th estimate and the $k$-th residual produced by a deterministic algorithm are denoted by $\breve{x}_k$ and $\breve{r}_k$. The $\ell_2$ norm is denoted without subscript ($\langle x , x \rangle = \|x\|^2$). The $A$-norm is denoted with subscript ($\langle A x, x \rangle = \|x\|_A^2$).

\subsection{Standard Krylov methods}

In the following, if no confusion is possible, we will simply denote $\mathcal{K}_k(A,r_0)$ by $\mathcal{K}_k$.
A \textit{Krylov solver} is rigorously characterized by the fulfilment of two conditions:
\begin{enumerate}
	\item For each $1 \leq k \leq n$, the estimate $\breve{x}_k$ of $x$ belongs to $x_0+\K_k$. \textit{(subspace-condition)}
	\item For each $1 \leq k \leq n$, the residual $\breve{r}_k = b - A \breve{x}_k$ is $\ell_2$-orthogonal to some $k$-dimensional subspace $\mathcal{L}_k$. \textit{(Petrov-Galerkin condition)}
\end{enumerate}
Two important examples of Krylov solvers are given by $\mathcal{L}_k = \K_k$, such as FOM and the CG algorithm, and $\mathcal{L}_k = A \K_k$, such as the GMRES algorithm (see~\cite{saad}). Iterative algorithms fulfiling the Petrov-Galerkin condition with $\mathcal{L}_k = \K_k$ are called orthogonal projection methods (O.P.M)~\cite[Proposition 5.2]{saad}. If $A$ is an SPD operator, it can be shown that
\begin{align} \label{pgc}
    \breve{r}_k \perp \K_k \iff \breve{x}_k = \argmin_{y \in \K_k} \| x - y \|_A^2
\end{align}
Iterative algorithms fulfiling the Petrov-Galerkin condition with $\mathcal{L}_k = A \K_k$ are a particular case of oblique projection methods~\cite[Proposition 5.3]{saad}. It can be shown that:
\begin{align*}
\breve{r}_k \perp A \K_k \iff \breve{x}_k = \argmin_{y \in \K_k} \| b - A y \|^2
\end{align*}

\Cref{algo:detarnoldi} qualifies as an OPM. To orthogonalize the residual $\breve{r}_k$ with respect to $\K_k$ at the $k$-th iteration, the whole basis $\breve{v}_1, \cdots, \breve{v}_k$ of $K_k$ is used (and so must be stored), which is why it is called full-orthogonalization method (FOM). The Petrov-Galerkin condition at iteration $j$ is enforced by solving a small reduced upper-Hessenberg system of dimension $j \times j$. Clearly, at each iteration $j$, the columns of $\breve{V}_{j}$ form an orthonormal basis for $\K_{j}$. Having the following relation for all $j \leq k$,
\begin{align*}
    \breve{h}_{j+1,j} \breve{v}_{j+1} = A \breve{v}_j - \breve{h}_{1,j} \breve{v}_1 - \cdots - \breve{h}_{j,j} \breve{v}_j    
\end{align*}
we get the Arnoldi relation and the Hessenberg-matrix formula
\begin{align}\label{projectedA}
    \begin{cases}
    A \breve{V}_k = \breve{V}_{k+1} \breve{H}_{k+1,k}\\
    \breve{V}_k^t A \breve{V}_k = \breve{H}_k.
    \end{cases}
\end{align}
which in turn can be used to show that~\Cref{algo:detarnoldi} gives $\breve{x}_k$ that satisfies the Petrov-Galerkin orthogonality condition $\breve{r}_k \perp \K_k$.

\begin{algorithm}
	\caption{Full Orthogonalization Method, a.k.a Arnoldi solver} \label{algo:detarnoldi}
	\KwInput{$A \in \mathbb{R}^{n\times n}, b\in \mathbb{R}^n, x_0 \in \mathbb{R}^n$}
	\KwOutput{$\breve{x}_k$, $\breve{V}_k$ an orthonormal basis of $\K_k$, $\breve{H}_k$ the matrix of the orthonormalization coefficients.}
	\SetKwComment{Comment}{/* }{ */}
	$r_0 \gets A x_0 - b$, $\beta \gets \|r_0\|$, $\breve{v}_1 \gets \beta^{-1} r_0$ \\
	\For{$j = 1, \hdots, k$} {
		$z \gets A \breve{v}_j$ \\
		\For{$i = 1, \hdots, j$} {
			$\breve{h}_{i, j} \gets \langle v_i, z\rangle$ \\
			$z = z - \breve{h}_{i, j} v_i$
		}
		$\breve{h}_{j+1,j} \gets \| z \|$ \\
		$\breve{v}_{j+1} \gets \breve{h}_{j+1,j}^{-1} z$ } 
	$\rho_k \gets \breve{H}_k^{-1} (\beta e_k) $ \\
	$\breve{x}_k \gets x_0 + \breve{V}_k \rho_k$\\
	Return $\breve{x}_k$, $\breve{V}_k$ and $\breve{H}_k$
\end{algorithm}

In~\Cref{algo:detarnoldi}, for orthonormalization of the basis we traditionally rely on the modified Gram-Schmidt (MGS) process. Alternatively, this task can also be performed with other stable methods such as the classical Gram-Schmidt process with re-orthogonalization (CGS2). While CGS2 requires twice as many flops as MGS, it can be more efficient from a performance standpoint.  Other techniques include Givens rotations and Householder transforms.

If $A$ is symmetric and non-singular, which is a common situation in applications, the relation~\cref{projectedA} implies that $\breve{H}_k$ is symmetric as well. The symmetry of $\breve{H}_k$, coupled with its upper-Hessenberg structure, implies that every element above the super-diagonal of this matrix is zero. Hence the matrix $\breve{H}_k$ is tridiagonal and symmetric. Recalling the definition of $\breve{H}_j$ at iteration $j$, this implies that for $i \leq j-2$, the inner products $\langle \breve{v}_i, z \rangle $ in line 5 of~\Cref{algo:detarnoldi} are all zero, i.e $A \breve{v}_j$ needs only to be orthogonalized with respect to $\breve{v}_j, \breve{v}_{j-1}$ Moreover, the coefficient $\langle A \breve{v}_j, \breve{v}_{j-1} \rangle = \breve{h}_{j,j-1}$ from the previous iteration needs not being computed again. Simplifying lines 4-6 of~\Cref{algo:detarnoldi} accordingly leads to the so-called Lanczos algorithm. The Lanczos algorithm requires only two inner products per iteration to perform the basis orthonormalization, which represents a significant cost reduction compared to~\Cref{algo:detarnoldi}. We stress out that the residual orthogonalization (the solving step), still requires all vectors of the basis to be stored. However, omitting the full orthogonalization against the previously computed vectors can come at the expense of poor accuracy under finite precision arithmetic. This phenomenon, known as loss of orthogonality, is described in C. Paige PhD thesis \cite{paige}, and is addressed by B.N Parlett and D.S Scott in~\cite{parlett}. Addressing this problem can significantly increase the computational cost of the method and make it nearly as expensive in terms of flops as FOM. 

The Lanczos algorithm can be improved by noticing that every residual $\breve{r}_k = b - A \breve{x}_k$ is linearly dependent to $\breve{v}_{k+1}$. Denoting the matrix of these successive residuals by $\breve{R}_k$, it follows that the matrix $\breve{R}_k^T A \breve{R}_k$ is also tridiagonal and symmetric. We can invoke a Crout factorization of this matrix, $\breve{R}_k^T A \breve{R}_k = \breve{L}_k \breve{D}_k \breve{L}_k^T$,
where $\breve{L}_k$ is lower triangular with only diagonal and sub-diagonal non-zeros, and $\breve{D}_k$ is diagonal. It can then be inferred that there exists an $A$-conjugate basis of $\K_k$ (namely, $\breve{L}_k^{-1} \breve{R}_k$) that satisfies a short recurrence relation with the residuals $\breve{r}_k$. This leads to the well-known CG given in~\Cref{algo:detcg}. 

\begin{algorithm}
	\caption{CG algorithm} \label{algo:detcg}
	\KwInput{$A \in \mathbb{R}^{n\times n}$ a SPD matrix, $b, x_0 \in \mathbb{R}^n$}
	\KwOutput{$\breve{x}_k \in \K_k$ such that $\breve{r}_k \perp \K_k$}
	$r_0 \gets A x_0 - b$, $p_0 = r_0$ \\
	\SetKwComment{Comment}{/* }{ */}
	\For{$j = 1, \hdots,  k$} {
		Compute $A \breve{p}_j$ \\
		$\breve{\gamma}_j \gets \frac{\langle \breve{r}_j, \breve{p}_j \rangle}{\langle A \breve{p}_j, \breve{p}_j \rangle}$ \\
		$\breve{x}_{j+1} \gets \breve{x}_j + \breve{\gamma}_j \breve{p}_j$ \\
		$\breve{r}_{j+1} \gets \breve{r}_j - \breve{\gamma}_j A \breve{p}_j$ \\
		$\breve{\delta}_{j+1} \gets - \frac{\langle \breve{r}_{j+1}, A \breve{p}_j \rangle}{\langle \breve{p}_j, A \breve{p}_j \rangle} \quad $  \\
		$\breve{p}_{j+1} \gets \breve{r}_{j+1} + \breve{\delta}_{j+1} \breve{p}_j$
	}
	
	Return $\breve{x}_k$ 
\end{algorithm}
Since the basis $\breve{p}_1, \hdots, \breve{p}_k$ is $A$-orthogonal, the restriction of $A$ to $\K_k$ in this basis is diagonal. This implies that once the contribution of one $\breve{p}_i$ to the estimate $\breve{x}_k$ is computed, it is no longer needed to refine the next estimate. For most of the applications, the CG algorithm is highly efficient because it only requires storage for three vectors and involves the computation of only three inner products per iteration, and has good numerical stability (see \cite{whycg}). For all these reasons, the CG algorithm is considered the \textit{nec plus ultra} among iterative methods. However it is still less stable than FOM, and as such the latter can be preferred for very ill-conditioned systems.

\subsection{Random sketching}

A matrix $\Omega \in \R^{ \ell \times n}$ is called an $\epsilon$-embedding of $\mathcal{V}_k$ if 
\begin{align}\label{subspaceembedding2}
\forall x \in \mathcal{V}_k, \, \, \left( 1 - \epsilon \right) \| x \|^2 \leq \| \Omega x \|^2 \leq \left(1 + \epsilon \right) \|x \|^2,
\end{align}
which, due to the parallelogram inequality, is equivalent to
\begin{align}\label{subspaceembedding1}
\forall x, y \in \mathcal{V}_k, \; \; | \langle \Omega x, \Omega y \rangle - \langle x, y \rangle | \leq \epsilon \|x\| \, \| y \|.
\end{align}
In this work, we will use $\Omega$ built with probabilistic techniques such that they have a high probability of being $\epsilon$-embeddings for arbitrary low-dimensional subspaces. These $\Omega$ are referred to as \textit{oblivious subspace embeddings} (OSEs) because they are generated without prior knowledge of the specific subspaces that they will embed.

\begin{definition}
	We say that $\Omega \in \R^{\ell \times n}$ is an oblivious subspace embedding with parameters $\epsilon, \delta, k$ if, given any $k$-dimensional subspace $\mathcal{V}_k \subset \R^n$, it satisfies~\cref{subspaceembedding1} with probability at least $1-\delta$.
\end{definition}
In the following, an oblivious subspace embedding with parameters $\epsilon, \delta, k$ will be simply called an $(\epsilon, \delta, k)$ OSE. There are several distributions that are guaranteed to satisfy the $(\epsilon, \delta, k)$ OSE property when the dimension $\ell$ is sufficiently large. Among them, we chose Gaussian matrices and SRHT as representatives. A Gaussian OSE has entries that are i.i.d. random variables drawn from $\mathcal{N}(0, \ell^{-1/2})$. It satisfies the $(\epsilon, \delta, k)$ OSE property if $\ell = \mathcal{O}(\epsilon^{-2} (k + \log \frac{1}{\delta}))$ (see proof in \cite{woodruff}). Sketching with such $\Omega$ can be beneficial in distributed or streamed computational environments where communication cost and passes over the data are of primary concern. However, in the classical sequential computational environment the benefit of sketching with this matrix can be mitigated. The SRHT OSE is defined as follows
\begin{align}
\Omega = \sqrt{\frac{n}{\ell}} P H D, \label{srht}
\end{align}
where $D \in \R^{n \times n}$ is a diagonal matrix of random signs, $H \in \R^{n \times n}$ is a Hadamard matrix (supposing that $n$ is a power of 2), and $P \in \R^{\ell \times n}$ is an uniform sampling matrix, i.e $\ell$ rows drawn from the identity matrix $I_n$. It is an ($\epsilon, \delta, k$) OSE if $\ell = \mathcal{O}(\epsilon^{-2} (k + \log \frac{n}{\delta}) \log\frac{k}{\delta})$, which is only slightly larger than the requirement for Gaussian matrices. It is worth noting that in both cases, the required sketching dimension is either independent of or only logarithmically dependent on $n$ and $\delta$. The advantage of SRHT is that it requires much fewer flops to compute $\Omega x$ than unstructured matrices. Specifically, with the standard implementation using the Walsh-Hadamard transform, it requires $O(n \log (n))$ flops, while with a more sophisticated implementation from~\cite{ailon}, it requires $O(n \log (\ell))$ flops. However, SRHT matrices are not as well-suited for distributed computing~\cite{yang}. In general, the choice of OSE should be based on the given computational architecture to achieve the most benefit.

Assume that $\Omega$ is an $\epsilon$-embedding for $\mathcal{V}_k$. We say that two vectors $v_1$ and $v_2$ from $\mathcal{V}_k$ are \textit{sketched orthogonal}, denoted by symbol $\perp^\Omega$,  if the sketches $\Omega v_1$ and $\Omega v_2$ are $\ell_2$-orthogonal. We say that $v_1, \cdots, v_k \in \mathcal{V}_k$ are a sketched orthonormal basis of $\mathcal{V}_k$ if and only if $\Omega v_1, \cdots, \Omega v_k$ are an orthonormal set of vectors of $\mathbb{R}^\ell$. Using the sketch-embedding property \Cref{subspaceembedding1}, it is clear that such a set of vector is linearly independant. They are also linearly independent from a numerical point of view. Indeed, denoting $V_k \in \R^{n\times k}$ the matrix which columns are formed by $v_1, \cdots , v_k$, we have the following result.
\begin{corollary}[Corollary 2.2 in~\cite{rgs}] If $\Omega$ is {an} $\epsilon$-embedding for $\mathcal{V}_k = \mathrm{range}(V_k)$, then the singular values of $V_k$ are bounded by $$ (1+\epsilon)^{-1/2} \sigma_{min}(\Omega V_k)  \leq \sigma_{min}(V_k) \leq \sigma_{max}(V_k) \leq  (1-\epsilon)^{-1/2} \sigma_{max}(\Omega V_k).$$
\end{corollary}
As a consequence, the condition number of a matrix $V_k \in \R^{n \times k}$ whose columns are sketched orthonormal (i.e such that the columns of $\Omega V_k \in \R^{\ell \times k}$ are $\ell_2$-orthonormal) is less than $\sqrt{3}$ for the typical $\epsilon = \undemi$.
Given a sketched orthonormal basis $v_1, \cdots, v_k$ of $\mathcal{V}_k$, we say that $z \in \R^\ell$ is sketched orthogonal to $\mathcal{V}_k$ if and only if the vectors $\Omega v_1, \cdots, \Omega v_k, \Omega z$ are $\ell_2$-orthogonal (assume that $\Omega$ is an $\epsilon$-embedding of $\mathcal{V}_k + \Span{z}$), and we write $z \perp^\Omega \mathcal{V}_k$. Finally, we denote by $\mathcal{P}_{\mathcal{V}_k}^\Omega$ a \textit{sketched orthogonal projector} onto $\mathcal{V}_k$, i.e., a linear mapping that satisfies 
$ (z - \mathcal{P}_{\mathcal{V}_k}^\Omega z)  \perp^\Omega \mathcal{V}_k$
for all $z \in \mathcal{V}_k$ (again, assume that $\Omega$ is an $\epsilon$-embedding of $\mathcal{V}_k + \Span{z}$). It is worth noting that the projector $\mathcal{P}_{\mathcal{V}_k}^\Omega$ satisfies
\begin{equation} \label{eq:skproj}
\mathcal{P}_{\mathcal{V}_k}^\Omega z = \argmin_{x \in \mathcal{V}_k} \| \Omega \left(z - x\right) \|.
\end{equation}

This observation serves as the foundation of the \textit{randomized Gram-Schmidt} (RGS) algorithm developed in~\cite{rgs}. Given a matrix $V_k$ and a subspace embedding $\Omega$ of $\mathcal{V}_k = \mathrm{range}(V_k)$, this algorithm orthogonalizes the columns of $V_k$ in a manner similar to the standard Gram-Schmidt algorithm. The distinguishing feature is that this orthogonalization is performed with respect to the sketched inner product $\langle \Omega \cdot, \Omega \cdot \rangle$ rather than the $\ell_2$-inner product. In detail, if $V_{j}$ denotes a matrix whose columns are the first $j$ orthogonalized columns of $V_k$, the RGS computes the next vector $v_{j+1}$ such that $v_{j+1} \perp^\Omega \Span{V_{j}}$. This allows the projection step of the Gram-Schmidt process to be performed on sketches rather than high-dimensional vectors, thus saving computational cost while preserving numerical stability. In particular, it was shown in~\cite{rgs} that RGS requires nearly half as many flops as MGS and four times fewer flops than CGS2, while being just as stable as these algorithms. Additionally, RGS can have a lower communication cost and requires fewer data passes.

The authors in~\cite{rgs} integrated the RGS algorithm into the Arnoldi iteration to construct a sketched orthonormal Krylov basis and then applied it to the GMRES method. It was shown that using a sketched orthonormal basis in GMRES is equivalent to minimizing the sketched residual error $\|\Omega r_k\|$.  Due to the $\epsilon$-embedding property of $\Omega$, it was directly deduced that the solution $x_k$ is a quasi-optimal minimizer of the residual error.

\section{Randomized Orthogonal Projection Methods}
In this section we prove our main result~\Cref{prop:quasi2}. First we define a class of algorithms that we refer to as ROPMs. 
\begin{definition}\label{def:ropm}
Let $\kmax \leq n$ be a positive integer. Let $\Omega$ be an $\epsilon$-embedding of $\K_{\kmax}$. A \textit{randomized orthogonal-projection method} (ROPM) is an iterative method producing a sequence of estimates $(x_k)_{k \leq \kmax}$ such that:
\begin{enumerate}
    \item For all $1 \leq k \leq \kmax$, the estimate $x_k$ belongs to $x_0 + \K_k$ (\textit{subspace condition})
    \item  For all $1 \leq k \leq \kmax$, the residual $r_k$ verifies the \textit{sketched Petrov-Galerkin condition}:
    \begin{align} \label{spgc}
        r_k \perp^\Omega \K_k.
    \end{align}
\end{enumerate}
\end{definition}
We recall that, since the Krylov subspace sequence is increasing, an $\epsilon$-embedding of $\K_{\kmax}$ is also an $\epsilon$-embedding of $\K_k$ for all $k \leq \kmax$. The first bound, given in~\Cref{prop:quasi1}, is derived directly from the $\epsilon$-embedding property. The second bound, given in~\Cref{prop:quasi2} is much tighter and describes accurately the behavior of the solver. Finally, we consider a specific ROPM, \Cref{algo:randarnoldi}, and give an indirect bound linking residuals produced by deterministic and randomized Arnoldi. 

\subsection{Quasi-optimality conditions}\label{section:quasi}
For better presentation assume that $x_0 = 0$. The extension to the cases with a general initial guess vector is straightforward. Assuming that $A$ is positive-definite, the quasi-optimality of ROPM solution can be characterized with the following result.
\begin{proposition} \label{prop:quasi1}
	Let $\Omega \in \mathbb{R}^{ \ell \times n}$ be an $\epsilon$-embedding of $\K_{k+1} + \mathrm{span}(x)$, with $\epsilon \leq \mathrm{cond}(A)^{-\frac{1}{2}}$. Assume that $A$ is positive-definite. Then the estimate $x_k \in \K_k$ produced by ROPM defined in~\Cref{def:ropm}, and $\breve{x}_k \in \K_k$ produced by standard OPM satisfy
	\begin{align} \label{quasi1}
	\|x - x_k \|_A \leq \frac{1 + \epsilon \, \mathrm{cond}(A)^\frac{1}{2}}{1 - \epsilon \, \mathrm{cond}(A)^\frac{1}{2}} \| x - \breve{x}_k \|_A.
	\end{align}
\end{proposition}

\begin{proof}
	By the $\epsilon$-embedding property of $\Omega$ we get
	\begin{align*}
	\| x - x_k \|_A^2 & =  \langle   x - x_k, A(x - x_k)  \rangle  \leq |  \langle   \Omega (x - x_k), \Omega A(x-x_k)  \rangle  + \epsilon \, \| x - x_k \| \cdot \| A (x - x_k) \| \\
	& \leq |  \langle   \Omega (x - \breve{x}_k), \Omega A (x-x_k)  \rangle  +  \langle   \Omega(\breve{x}_k - x_k) , \Omega A (x - x_k)  \rangle  | + \epsilon \, \| x - x_k \| \cdot \| A (x - x_k) \|
	\end{align*}
	Due to the fact that $\breve{x}_k - x_k \in \K_k$ and the sketched Petrov-Galerkin projection property, the second term is null. Then, we have
	\begin{align*}
	\| x - x_k \|_A^2 & \leq |  \langle   \Omega (x - \breve{x}_k), \Omega A (x-x_k)  \rangle   | + \epsilon \, \| x - x_k \| \cdot \| A (x - x_k) \| \\
	& \leq  \langle   x - \breve{x}_k, A (x - x_k)  \rangle  + \epsilon \, \| x - \breve{x}_k \| \cdot \| A (x - x_k) \| + \epsilon \, \| x - x_k \| \cdot \| A (x - x_k) \| \\
	& \leq  \| x - \breve{x}_k \|_A \cdot \| x - x_k \|_A + \epsilon \, \mathrm{cond}(A)^\frac{1}{2} \| x - \breve{x}_k \|_A \cdot \|x-x_k\|_A + \epsilon \, \mathrm{cond}(A)^\frac{1}{2} \|x - x_k\|_A^2
	\end{align*}
	Dividing by $\| x - x_k \|_A$ on both sides yields~\cref{quasi1}.
\end{proof}
\Cref{prop:quasi1} states that the $A$-norm error of ${x}_k$ is close to optimal, given that $\Omega$ is an $\epsilon$-embedding with $\epsilon < \mathrm{cond}(A)^{-\frac{1}{2}}$. This property can be satisfied with high probability if $\Omega$ is an $(\epsilon, \delta, k+2)$ OSE. However, for typical OSEs, this condition would require the usage of sketching dimension $\ell \geq \mathcal{O}( \mathrm{cond}(A) \times k )$ (see \cite{woodruff}), which is fine for very well-conditioned systems, but unrealistic for even moderately ill-conditioned systems. Our thorough numerical experiments, on the other hand, have shown that ROPM should achieve comparable accuracy to the classical OPM, even when using sketching matrices with $\ell = 2k$ or $4k$. However, the convergence of ROPM in this scenario may exhibit an irregular behavior, characterized by occasional spikes in error (sometimes dramatic) under unfavorable conditions for sketching, as demonstrated in~\Cref{fig:bound2}. This figure depicts the convergence of the OPM and ROPM solutions for the \textit{Si41Ge41H72} system of size $n \approx 200000$, taken from~\cite{suitesparse}, and shifted by $1.2138I$ to ensure positive-definiteness and a condition number of $\approx 10^4$. The right-hand-side vector here was taken as a normalized Gaussian vector. The solution of such a shifted system can be relevant, for instance, in the context of inverse iteration for finding the smallest eigenvalue(s) of the \textit{Si41Ge41H72} operator. In the ROPM, the sketching matrix $\Omega$ was taken as SRHT of varying size. \Cref{fig:bound2}~confirms sufficient convergence of ROPM even for $\ell = 2k$ sketching dimension. However, we  clearly reveal dramatic spikes of the error at the beginning of the descent of the error. While during the descent, the behavior of the convergence becomes more favorable. This phenomenon has motivated further exploration into the properties of the linear system that impact the accuracy of ROPM. 

The following theorem provides more general error bound for ROPM.
\begin{theorem} \label{prop:quasi2}
	Assume that $A$ is positive-definite. We have
	\begin{align} \label{eq:quasi2}
	\| x - x_k\|_A \leq \left (1+\alpha^2_k \beta^2_k \right)^{\frac{1}{2}} \| x - \breve{x}_k\|_A, 
	\end{align}
	where  $x_k \in \K_k$ and $\breve{x}_k \in \K_k$ are the sketched and the classical Petrov-Galerkin projections, respectively, and
	\begin{subequations}
		\begin{align}
		\alpha_k &:=  \frac{\langle {x} - \breve{x}_{k-1}, \Pk  A \breve{v}_k \rangle }{\langle x - \breve{x}_{k-1},\Pks A \breve{v}_k \rangle}, \\
		\beta_k &:= \|A^{-\frac{1}{2}}\Pks \breve{v}_{k+1}\|  \langle \breve{v}_{k+1}, A \breve{v}_{k} \rangle  \frac{\langle \breve{v}_{k}, \breve{x}_{k} \rangle}{\| x - \breve{x}_k\|_A},
		\end{align}
	\end{subequations}
	and where $\breve{v}_k$ is a unit vector spanning the range of $(I - \mathcal{P}_{\K_{k-1}}) \mathcal{P}_{\K_{k}}$.
	\begin{proof}
	    Note that for all $k$, the vector $\breve{v}_k$ is also the $k$-th vector built by Arnoldi iteration. It holds that  
		\begin{align*}
		\| x_k - \breve{x}_k\|_A^2 &=  \langle  x_k - \breve{x}_k, A (x_k - \breve{x}_k) \rangle =\langle  x_k - \breve{x}_k, A (x_k - x) \rangle = \langle  x_k - \breve{x}_k,  (I - \Pks) A (x_k - x) \rangle \\ 
		&= -\langle  x_k - \breve{x}_k,  \Pks (\breve{v}_{k+1} \breve{v}^t_{k+1}) A (x_k - x) \rangle
		= -\langle x_k - \breve{x}_k,  \Pks \breve{v}_{k+1}\rangle |\breve{v}^t_{k+1} A \breve{v}_{k}| |\breve{v}^t_{k} x_k | \\
		&\leq \| x_k - \breve{x}_k\|_A \|A^{-\frac{1}{2}} \Pks \breve{v}_{k+1}\| |\breve{v}^t_{k+1} A \breve{v}_{k}| |\breve{v}^t_{k} x_k |.
		\end{align*}
		Furthermore, from the relation 
		$$ (x-\breve{x}_{k-1})^t \Pk A \breve{x}_k  = (x-\breve{x}_{k-1})^t \Pks A x_k $$
		we deduce that  
		$$  (x-\breve{x}_{k-1})^t \Pk A \breve{v}_k \breve{v}_k^t \breve{x}_k = (x-\breve{x}_{k-1})^t \Pks A \breve{v}_k \breve{v}_k^t x_k$$
		which in turn yields the following relation 
		$$ \breve{v}^t_{k} x_k  = \breve{v}^t_{k} \breve{x}_k \frac{\langle {x} - \breve{x}_{k-1}, \Pk  A \breve{v}_k \rangle }{\langle x - \breve{x}_{k-1},\Pks A \breve{v}_k \rangle}.$$
		By substituting this relation to the derived earlier bound for $\| x_k - \breve{x}_k\|_A^2$ and dividing both sides by $\| x_k - \breve{x}_k\|_A$, we obtain: 
		$$\|  x_k - \breve{x}_k\|_A \leq \|A^{-\frac{1}{2}} \Pks \breve{v}_{k+1}\| |\breve{v}^t_{k+1} A \breve{v}_{k}| |\breve{v}^t_{k} \breve{x}_k| \frac{|\langle {x} - \breve{x}_{k-1}, \Pk  A \breve{v}_k \rangle| }{|\langle x - \breve{x}_{k-1},\Pks A \breve{v}_k \rangle|} = |\alpha_k||\beta_k| \| x - \breve{x}_k\|_A.$$
		The statement of the theorem then follows by the Pythagorean equality.
	\end{proof}
\end{theorem}

The validity of the bound in~\Cref{prop:quasi2} was substantiated through a set of numerical experiments. Our results showed that~\cref{eq:quasi2} accurately captures the convergence of~\Cref{algo:randarnoldi} error (which is a particular case of ROPM as shown in the next section), as can be seen from~\Cref{fig:alpha_beta}. Additionally, it can be seen that the error bound and the actual error become practically indistinguishable as the OPM (\Cref{algo:detarnoldi}) and ROPM solutions enter the region of steep convergence. More details related to the numerical validation of this bound can be found in~\Cref{section:results}.

The result stated in \Cref{prop:quasi2} holds without any assumptions on the sketching matrix used in the sketched Petrov-Galerkin projection. It implies that the accuracy of the ROPM can be ensured as long as $|\Pks \breve{v}_{k+1}|$ is sufficiently small. This conclusion is based on the following observation: 
\begin{subequations} \label{eq:alphabeta}
	\begin{align}
	|1 - \alpha_k^{-1}| &= \frac{|\langle {x} - \breve{x}_{k-1}, (\Pk-\Pks)  A \breve{v}_k \rangle|}{|\langle x - \breve{x}_{k-1},\Pk A \breve{v}_k \rangle|}  \leq \|\Pks  \breve{v}_{k+1}\| \frac{\| \Pk {x} - \breve{x}_{k-1}\| | \breve{v}_{k+1}^t A \breve{v}_k| }{|\langle x - \breve{x}_{k-1},\Pk A \breve{v}_k \rangle|} \label{eq:alphak} \\
	\intertext{and} 
	|\beta_k| &\leq \|\Pks \breve{v}_{k+1}\|  \|A^{-\frac{1}{2}}\|   |\langle \breve{v}_{k+1}, A \breve{v}_{k} \rangle|  \frac{|\langle \breve{v}_{k}, \breve{x}_{k} \rangle|}{\| x - \breve{x}_k\|_A}. \label{eq:betak}
	\end{align}
\end{subequations}
Relations \cref{eq:alphabeta} show that when $\|\Pks \breve{v}_{k+1}\|$ is sufficiently small, we have $\alpha_k \approx 1$ and $\beta_k \approx 0$. Note that if $\Omega$ is an $\varepsilon$-embedding for $\K_{k+1}$, then $\|\Pks \breve{v}_{k+1}\| = \mathcal{O}(\varepsilon)$, which combined with~\cref{eq:alphabeta} implies the accuracy of ROPM when using an OSE of sufficiently large size. It is worth noting that the bound in~\cref{eq:betak} for $\beta_k$ may be overly conservative, as $\Pks \breve{v}_{k+1}$ is expected to be equally represented by all eigenvectors of $\Pk A^{-1} \Pk$, not just the dominant eigenvectors.

Finally, it has been revealed by our numerical experiments that the coefficients $\alpha_k, \beta_k$ capture two features of ROPM solver behavior. The spikes seem to be captured by the coefficient $\alpha_k$. A more durable increase of the quasi-optimality constant seem to be captured by the coefficient $\beta_k$. This observation is supported by the results shown in~\Cref{fig:alpha_beta}, where it is evident that the spikes in the error correspond to the spikes in the value of $\alpha_k$. Furthermore, at the end of the convergence, the bases of the spikes no longer fit exactly the OPM error, which seems to be captured by a durable increase of the coefficient $\beta_k$.  Based on~\cref{eq:alphak}, it can be hypothesized that the accuracy of the ROPM solution $x_k$ is largely dependent on the orthogonality between the error ${x} - \breve{x}_{k-1}$ from the $k-1$-th iteration and the image $\Pk A \breve{v}_k$ of the $k$-th standard Krylov vector $\breve{v}_k$. Clearly this phenomenon shall depend on the properties of the linear system that is being solved. Its analysis is left for future research.

\begin{figure}[H] \label{fig:ROPM}
	\begin{subfigure}[t]{.47\textwidth}
		\centering
		\includegraphics[width=\linewidth]{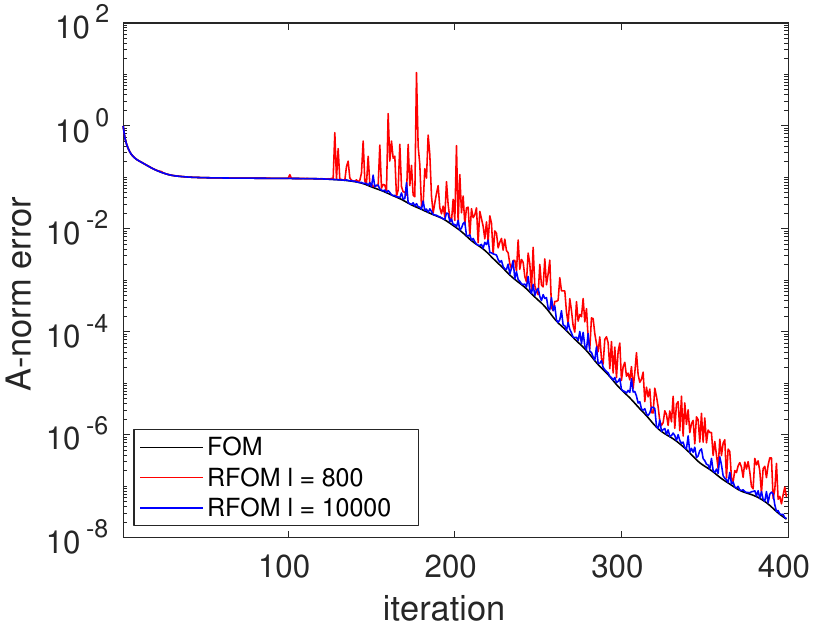}
		\caption{RFOM convergence.}\label{fig:bound2}
	\end{subfigure}
	\hfill 
	\begin{subfigure}[t]{.47\textwidth}
		\centering
		\includegraphics[width=\linewidth]{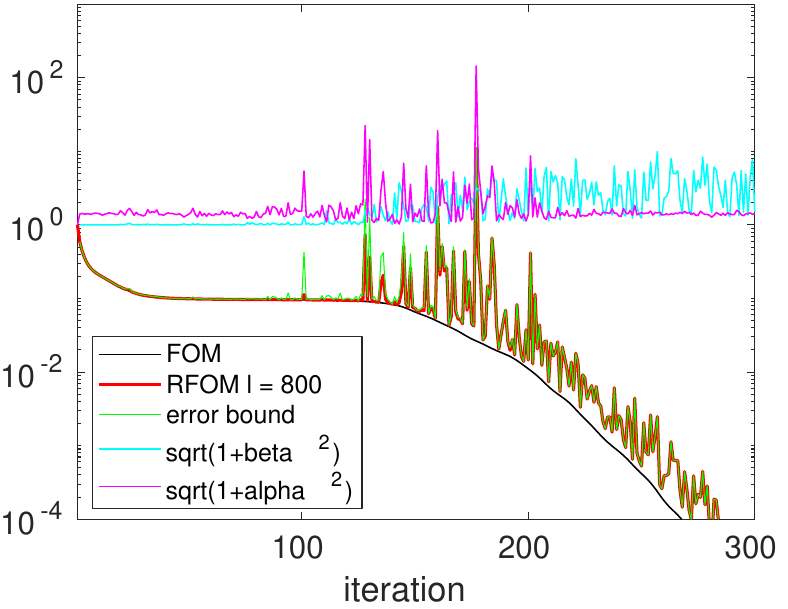}
		\caption{Effect of $\alpha_k$ and $\beta_k$ on irregularity of RFOM.}\label{fig:alpha_beta}
	\end{subfigure}
    \caption{RFOM convergence for the shifted \textit{Si41Ge41H72} system.} 
\end{figure}

Given a factorization $A = B^t B$, it is possible to split the sketched Petrov Galerkine condition and require instead that $B (x - x_k) \perp^\Omega B \K_k$. If $x_k$ verifies this condition, it is straightforward to derive
$\|x - x_k \|_A \leq \frac{1 + \epsilon}{1- \epsilon} \| x - \breve{x}_k \|_A.$
We observed that imposing this split sketched Petrov-Galerkin condition in our experiment, the resulting convergence rate is identical to that of OPM solver, and doesn't feature any spike on matrices \textit{G-clust5-s25} and \textit{El3D} presented in~\Cref{section:results}. However, straightforwardly imposing this condition requires access to $B x$.

\subsection{Randomized Arnoldi}\label{section:rarnoldi}

In this section, we consider a specific ROPM, the randomized Arnoldi solver~\Cref{algo:randarnoldi}. It is also denoted as RFOM since it performs a (sketched) orthogonalization against the full Krylov basis at each iteration. We give an indirect bound that links the residual of~\Cref{algo:detarnoldi} to that of~\Cref{algo:randarnoldi}. Numerical experiments comparing those two algorithms are presented in~\Cref{section:results}. First, the following lemma shows that~\Cref{algo:randarnoldi} satisfies~\Cref{def:ropm}.

\begin{lemma}\label{lemma:randarnoldiropm}
    In infinite precision arithmetics, if $\Omega$ is an $\epsilon$-embedding of $\K_{\kmax}$, then~\Cref{algo:randarnoldi} is an ROPM and does not break before the maximal Krylov subspace is generated.
\end{lemma}
\begin{proof} 
We show that the sketched normalization of $z$ (line 10) is not a division by $0$ until the maximum Krylov subspace is generated, and that the sketched Petrov-Galerkin condition is met. 

Since the Krylov subspace sequence is strictly increasing until the maximum Krylov subspace is generated, the vector $z$ obtained after the orthogonalization (line 8) is a non-zero vector of $\K_{\kmax}$. $\Omega$ being a subspace embedding of $\K_{\kmax}$,
$$ \| z \| > 0 \implies \frac{1}{1+\epsilon} \| \Omega z \| > 0 \implies \| \Omega z \| > 0,$$
hence line 10 is safe until the maximum Krylov subspace is generated. It is then straightforward to show by induction that the vectors $v_1, \cdots, v_k$ form a sketched orthonormal basis of $\K_k$ and that we get the randomized counterpart of the Arnoldi relation:
\begin{align*}
\begin{cases}
A V_k = V_{k+1} H_{k+1, k}\\
\left( \Omega V_k \right)^t \Omega A V_k = H_k.
\end{cases}
\end{align*}
As to the ROPM characterization, the sketched Petrov-Galerkin condition \eqref{spgc} is implied by
\begin{align} \label{spgc_basis}
    \left( \Omega V_k \right)^t \Omega r_k = 0_{\mathbb{R}^k}. 
\end{align}
Expressing $x_k$ as $x_0 + V_k \rho_k$,
\begin{align*}
    \left( \Omega V_k \right)^t \Omega r_k = 0_{\mathbb{R}^k} \iff \left( \Omega V_k \right)^t \Omega \left( r_0 - A V_k \rho_k \right) = 0  \iff H_k \rho_k = \left( \Omega V_k \right)^t \Omega r_0.
\end{align*}
For $k \leq \kmax$, the vectors $\Omega v_1, \cdots \Omega v_k$ form an orthonormal basis of $\K_k$, yielding for any $\rho_k \in \R^k$
$$ (\Omega V_k)^t \Omega A V_k \rho_k = 0 \implies \Omega  A V_k \rho_k = 0 \implies (1-\epsilon) \|A V_k \rho_k \|^2 = 0 \implies A V_k \rho_k = 0 \implies \rho_k = 0,$$
hence $H_k$ is non singular. We infer
$$ H_k \rho_k = (\Omega V_k)^t \Omega r_0 \iff \rho_k = H_k^{-1} (\Omega V_k)^t \Omega r_0.$$
\end{proof}

The randomized FOM method based on the RGS algorithm was presented in~\cite[Remark 4.1]{rgs}, although it was intended for solving linear systems with multiple right-hand sides. In the case of a single linear system, it can be simplified to the algorithm below.

\begin{algorithm}
	\caption{Randomized FOM (based on RGS~\cite{rgs}), a.k.a Randomized Arnoldi Solver} \label{algo:randarnoldi}
	\KwInput{$A \in \mathbb{R}^{n\times n}, b\in \mathbb{R}^n, x_0 \in \mathbb{R}^n$, and $\Omega \in \mathbb{R}^{ \ell \times n}$}
	\KwOutput{${x}_k$, ${V}_k$ a sketched orthonormal basis of $\K_k$, ${H}_k$ the matrix of the orthonormalization coefficients.}
	\SetKwComment{Comment}{/* }{ */}
	$r_0 \gets A x_0 - b$, $\beta \gets \|\Omega r_0\|$, ${v}_1 \gets \beta^{-1} r_0$, ${s}_1 \gets \beta^{-1} \Omega r_0$ \\
	\For{$j = 1, \hdots, k$} {
		$z \gets A {v}_j$ \\
		$p \gets \Omega z$ \\
		\For{$i = 1, \hdots, j$} {
			${h}_{i, j} \gets \langle s_i, p\rangle$ \\
			$p = p - {h}_{i, j} s_i$
		}
		$z \gets z - V_j H_{[1:j, \, j]}$ \\
		$s' \gets \Omega  z $, $h_{j+1,j} \gets \| s' \|$ \\
		${v}_{j+1} \gets {h}_{j+1,j}^{-1} z$, ${s}_{j+1} \gets {h}_{j+1,j}^{-1} s'$ }
	$\rho_k \gets {H}_k^{-1} (\beta e_k) $ \\
	${x}_k \gets x_0 + {V}_k \rho_k$\\
	Return ${x}_k$, ${V}_k$ and ${H}_k$
\end{algorithm}
In steps 5-7 of~\Cref{algo:randarnoldi}, we orthogonalized $\Omega A v_j$ to $\mathrm{range}(\Omega V_{j})$ with the MGS projector, in order to maintain a direct correspondence with~\Cref{algo:detarnoldi}. However, this orthogonalization could also be performed with more stable methods~\cite{rgs,BRGS}. For example, the orthogonalization coefficients $H_{[1:j, j]}$ can be obtained by solving the following low-dimensional least-squares problem:
$$ H_{[1:j, j]} = \argmin_{x \in \R^j} {\|\Omega (V_{j} x - A v_j) \|} $$
using the Householder method. According to~\cite{BRGS}, the estimate ${x}_k$ obtained through the randomized FOM satisfies the sketched variant of the Petrov-Galerkin orthogonality condition: $${r}_k \perp^\Omega \mathcal{K}_k.$$

In order to compare the residual produced by randomized Arnoldi to its deterministic counterpart in \Cref{prop:diffresbound}, we derive the following lemma.
\begin{lemma}
Let $H_k$, $V_k$, $x_k$ be the Hessenberg matrix, Krylov vectors and estimate produced by $k$ iterations of~\Cref{algo:randarnoldi}, respectively. Let us suppose $x_0 = 0_{\mathbb{R}^n}$. Then we have 
\begin{align*}
    A x_k = r_0 + \|\Omega r_0 \| s_{k,1} v_{k+1} = b + \| \Omega b\| s_{k,1} v_{k+1},
\end{align*}
where $s_k^t$ is the product of the scalar $h_{k+1,k}$ and the last row of $H_k^{-1}$, hence $s_{k,1}$ is the product of the scalar $h_{k+1,k}$ and the bottom-left corner of the matrix $H_k^{-1}$.
We derive the similar equation involving the byproducts of~\Cref{algo:detarnoldi}:
\begin{align*}
    A \breve{x_k} = r_0 + \| r_0 \| \breve{s}_{k,1} \breve{v}_{k+1} = b + \| b \| \breve{s}_{k,1} + \breve{v}_{k+1}.
\end{align*}
\end{lemma}
\begin{proof} 
We show it for $A x_k$, the computations are the same for the deterministic counterpart.
\begin{align*}
    A x_k = A \left[V_k H_k^{-1} (\Omega V_k)^t \Omega r_0\right] = V_{k+1} H_{k+1,k} \left[H_k^{-1} (\Omega V_k)^t \Omega r_0 \right] = V_{k+1} \left[\begin{matrix} I_k \\ \hline s_k^t \end{matrix}\right] \| \Omega r_0\| e_1 = r_0 + \|\Omega r_0\| s_{k,1} v_{k+1}
\end{align*}
\end{proof}
We can now derive an easy-to-compute indirect bound.
\begin{proposition}
\label{prop:diffresbound}
Let $\Omega \in \R^{\ell \times n }$ be an $\epsilon$-$\ell_2$subspace embedding of $\K_{k+1}$. If $x_k \in \K_k$ satisfies the sketched Petrov-Galerkin condition \eqref{spgc}, and if $\breve{x}_k$ satisfies the Petrov-Galerkin condition $\eqref{pgc}$,
\begin{align}\label{rescomparison}
    & \| r_k \| \leq \| A x_0 \| + (1 + \epsilon) \, \| \Omega r_0 \| \cdot |s_{k,1}|, \\
    & \| r_k \| \leq \| \breve{r}_k \| + \|r_0\| \left[ | \breve{s}_{k,1} | + \sqrt{\frac{1+\epsilon}{1 - \epsilon}} \, | s_{k,1} | \right].
\end{align}
\end{proposition}
\begin{proof}
    The first bound is a direct consequence of the previous lemma. For the second one, subtracting the deterministic and randomized counterparts, we get
\begin{align*}
    \breve{r}_k - r_k = A x_k - A \breve{x}_k = \| \Omega r_0 \| s_{k,1} v_{k+1} -  \| r_0 \| \breve{s}_{k,1} \breve{v}_{k+1}.
\end{align*}
The norm of the latter verifies:
\begin{align*}
    \| A x_k - A \breve{x}_k \| & \leq \|r_0\| \, |s_{k,1}| + \sqrt{1+\epsilon} \, \| r_0 \| \, | \breve{s}_{k,1} | \frac{1}{\sqrt{1-\epsilon}}.
\end{align*}
The result follows by factoring by $\|r_0\|$.
\end{proof}
The first bound is interesting only if the first guess $x_0$ is set to zero. The coefficient $s_{k,1}$ (resp. $\breve{s}_{k,1}$) is the bottom-right corner $h_{k+1,k}$ (resp $\breve{h}_{k+1,k}$) of upper-Hessenberg matrix $H_k$ (resp. $\breve{H}_k$) multiplied by bottom-left corner of its inverse. In our experiments, this coefficient is small and in this case the bound is tight. We illustrate this in~\Cref{fig:resbound} on one of the matrices from our test set (see~\Cref{table}), \textit{Dubcova3} from~\cite{suitesparse}. We assume that the embedding was successful with $\epsilon = \undemi$.

\begin{figure}
    \hspace{5.95cm}\includegraphics[width=.33\linewidth]{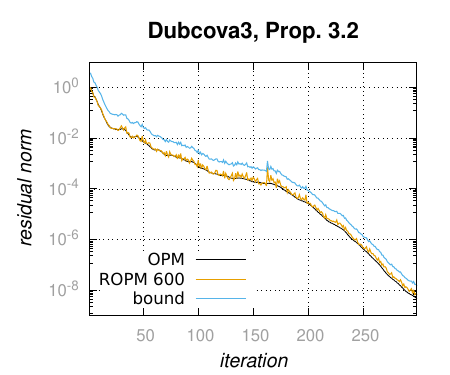}
  \caption{Residual bound}
  \label{fig:resbound}
\end{figure}

\section{Short-recurrence and randomization} \label{section:CG}
In this section we consider the problem of building a sketched o  rthogonal basis of the Krylov subspace with constant storage cost for symmetric matrices. We notice the appearance of noise in the Hessenberg matrix obtained from randomized Arnoldi in~\Cref{algo:randarnoldi}. We see experimentally that the distribution of this noise is close to Gaussian distribution. We explore the implications of the empirical claim that this noise is Gaussian, with regards to constant-storage cost orthonormalization processes. We consider an approximation of CG through randomization in~\Cref{algo:arCG} that avoids one synchronization. While in general~\Cref{algo:arCG} is not proven to be stable, the numerical experiments presented in~\Cref{section:results} highlight situations in which its rate of convergence is similar to that of~\Cref{algo:detcg}.

\subsection{Impact of randomization based on sampling}\label{section:symmetry}

Denoting $V_k$ the matrix formed by the vectors $v_1, \hdots , v_{k}$ and $H_k$ the matrix formed by the $k$ first rows of $H_{k+1,k}$ produced by~\Cref{algo:randarnoldi}, we have the following implication:
\begin{align}\label{randarnoldirelation}
\begin{cases}
    A V_k = V_{k+1} H_{k+1,k} \\
    ( \Omega V_k )^t \Omega A V_k = H_k
\end{cases}
\end{align}
The symmetry of $H_k$ is lost, as the left-hand side of second equation in~\cref{randarnoldirelation} isn't symmetric anymore in general.
\begin{proposition} \label{prop:noCommutation}
Suppose $A \in \R^{n \times n}$ is symmetric and has $\ell+q$ distinct eigenvalues, where $\ell \ll n$. Let $\Omega \in \R^{\ell \times n}$ be an arbitrary matrix. If $\Omega^t \Omega$ and $A$ commute, then at least $q$ eigenvectors of $A$ fall into $\Omega$'s kernel.
\end{proposition}
\begin{proof}
Since $A$ is symmetric, its spectral decomposition is $A = P \Lambda P^t$, where $P$ is an orthogonal matrix formed by the eigenvectors of $A$ and $\Lambda$ is a diagonal matrix formed by the corresponding eigenvalues. Without loss of generality, suppose that $m_1, \cdots m_{\ell+q} \in \mathbb{N}^*$, \, $m_1 + \cdots + m_{\ell+q} = n$ and $\Lambda = \text{BlockDiag}\left( \lambda_1 I_{m_1}, \, \cdots, \, \lambda_{\ell + 1} I_{m_{\ell+q}} \right)$, where $\text{BlockDiag}$ denotes a block-diagonal matrix with blocks given as arguments and $I_{m_i}$ is the identity matrix of size $m_i$, for $1 \leq i \leq \ell + q$. Any matrix that commutes with $A$  also commutes with $\Lambda$. Hence any matrix that commutes with $A$ is block diagonal with blocks of consecutive size $m_1, m_2, \cdots , m_{\ell+q}$. This yields $P^t \Omega^t \Omega P = \text{BlockDiag} \left( B_{m_1}, \cdots, B_{m_{\ell + q}} \right)$, where $B_{m_i}$ are matrices of size $m_i \times m_i, \; 1 \leq i \leq \ell + q$. Let us denote the columns of $P$ as $p_1, \cdots , p_{n}$. Let us take $i_1 \in \{1, \cdots, m_1\}$, $i_2 \in \{m_1 +1 , \ldots m_1 + m_2 \}$, $\ldots$ $i_{\ell+q} \in \{m_1 + ... + m_\ell + 1 , \cdots , m_1 + \ldots + m_\ell + m_{\ell+q}\}$. Let us denote $\tilde{P}$ the tall-and-skinny matrix formed by the columns $p_{i_1}, \cdots, p_{i_{\ell+q}}$. Then the previous equation shows that $\left(\Omega \tilde{P}\right)^t  \Omega \tilde{P}$ is diagonal. Since $\ell+q$ non-zero vectors cannot be $\ell_2$-orthogonal in $\mathbb{R}^\ell$, at least $q$ of those vectors are null, i.e at least $q$ eigen-vectors of $A$ fall into $\Omega$'s kernel.
\end{proof}
The typical SPD operator having $\Theta(n)$ eigenvalues, this theorem disqualifies the natural sufficient condition such that $x,y \mapsto \langle \Omega A x, \Omega y \rangle$ is symmetric. Then with our choice of OSE matrices (Gaussian or SRHT), we study the upper-Hessenberg matrix $H_k$ obtained from~\Cref{algo:randarnoldi}. Denoting $\Gamma_k = H_k - \breve{H}_k$, \Cref{fig:noise_4096,fig:GaussianNoise_overall,fig:GaussianNoise_one} suggest that the coefficients above the super-diagonal of $H_k$ have a distribution close to that of a gaussian white noise.

\Cref{fig:GaussianNoise} showcases the distribution of the coefficients above the superdiagonal of the Hessenberg matrix obtained from~\Cref{algo:randarnoldi}. In~\Cref{fig:noise_4096} we see noise above the superdiagonal that resembles gaussian noise. \Cref{fig:GaussianNoise_overall} shows the statistic-distribution of all the coefficients above the super-diagonal on one simulation (we choose intervals whose length is half a standard deviation). \Cref{fig:GaussianNoise_one} shows the statistic-distribution of one coefficient above the super-diagonal drawn on multiple distributions. In our experiments, this phenomenon occurs already with low sampling size. We stress out that for the case of SRHT and Gaussian OSE, one can write 
\begin{align*}
\gamma_{i,j} = \left[ \left( \frac{1}{\ell}\sum_{s=1}^\ell W_{i,j,s} \right) - \mathbb{E}\left[W_{i,j,1} \, | \, v_i, \, v_j \right] \right] + \mu_{i,j}
\end{align*}
where $\mu_{i,j} = \langle A v_i, v_j \rangle - \langle A \breve{v}_i, \breve{v}_j \rangle$ and the $W_{i,j,s}$ are identically distributed random variables. If they were independent from each other and independent from $v_i, v_j$, this expression would simplify and qualify for the central limit theorem.

\begin{figure}
  \begin{subfigure}[t]{.33\textwidth}
    \centering
    \includegraphics[width=\linewidth]{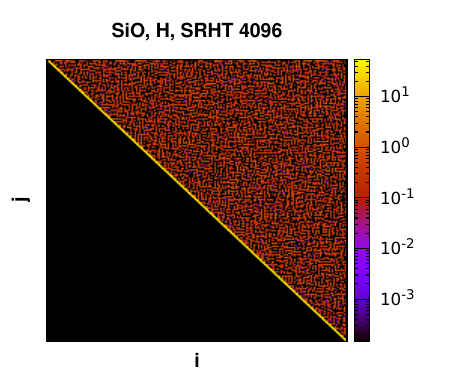}
    \caption{4096 sampling size}\label{fig:noise_4096}
  \end{subfigure}
  \begin{subfigure}[t]{.33\textwidth}
    \centering
    \includegraphics[width=\linewidth]{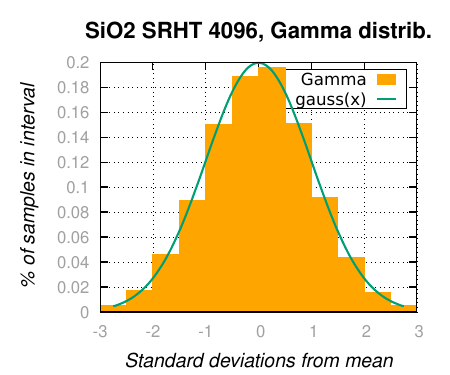}
    \caption{Overall distribution}\label{fig:GaussianNoise_overall}
  \end{subfigure}
  \begin{subfigure}[t]{.33\textwidth}
    \centering
    \includegraphics[width=\linewidth]{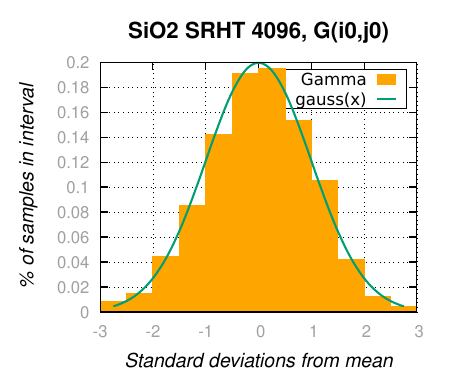}
    \caption{$\Gamma_{i_0, j_0}$ distribution}\label{fig:GaussianNoise_one}
  \end{subfigure}
  \caption{Distribution of $\Gamma_k$ (noise above the superdiagonal)}
  \label{fig:GaussianNoise}
\end{figure}

We investigate the empirical implications of considering the entries $h_{i,j}$ of $H_k$ above the superdiagonal as gaussian variables. One cannot ignore these coefficients, as
\begin{align*}
    \| \Omega \mathcal{P}^\Omega_{\K_k}( A v_k - h_{k,k} v_k - h_{k-1,k} v_{k-1})\|^2 = \sum_{j = 1}^{k-2} h_{i,k}^2. 
\end{align*}
This equation means that, if we only use the last two basis vectors to try and sketched orthonormalize $A v_k$ against $\K_k$, the vector that remains has its squared sketched norm greater than that of a $k-2$-dimensional gaussian vector. 

We now consider the case of an algorithm allowing $r$-terms recurrence. Let $\K_k$ be the Krylov subspace of index $k \in \{1, \cdots n-1\}$, let $\Omega \in \R^{\ell \times n}$ be either SRHT or Gaussian, and suppose that $\Omega$ is an $\epsilon$-embedding of $\K_k$. Suppose that $v_1, \cdots, v_k$ form a sketched orthonormal basis of $\K_k$, with $\| \Omega r_0\| v_1 = r_0$. Let us suppose that an algorithm, having built this basis, must proceed and build $v_{k+1}$ with only $r$ vectors $u_1, \cdots, u_r$ of $\R^n$, $r \ll k$ of $\R^n$ in memory, using only linear combinations of these $r$ vectors, $A$ and $\Omega$. The number of sketches that could be stored by this algorithm is not restricted. 
\begin{lemma}
Let $r$ be a positive integer such that $2r \ll k-2$. Let $u_1, \cdots, u_r \in \K_k$. Let us denote $U_{2r} \in \R^{2r}$ the matrix whose columns are formed by $u_1, \cdots u_r, A u_1, \cdots A u_r$. If $(\Omega V_{k-2})^t \Omega A v_k$ is a Gaussian vector with entries following $\mathcal{N}(\mu, \sigma^2)$ distribution, then
\begin{align*}
    \min_{x \in \R^{2r}} \| \Omega (\mathcal{P}_{k-2}^\Omega A v_k - U_{2r} x) \|^2 = \| g \|^2
\end{align*}
where $g$ is a random gaussian vector of dimension $k - 2r - 2$ with entries also following $\mathcal{N}(\mu, \sigma^2)$.
\end{lemma}
\begin{proof}
We detail the proof for the case where $\mathrm{dim} \; \mathrm{Range} (U_{2r}) \cap \mathrm{Span}\langle v_{k-1}, v_k \rangle = 0$. We recall that $r$ is taken such that $2r \ll k-2$. First, apply a sketched orthonormalization process to $u_1, \cdots u_r, A u_1, \cdots A u_r$ and obtain the vectors $w_1, \cdots w_s$, where $s$ is at most $2r$. Complete this basis in a sketched orthonormal basis of $\K_{k-2}$ and obtain the vectors $w_1, \cdots w_{k-2}$. We now have two orthonormal bases of $\Omega \K_{k-2}$, namely the vectors $w_i$ and the vectors $v_i$ built by~\Cref{algo:randarnoldi}. Denote $f \in \mathcal{L}(\Omega \K_{k-2})$ the linear operator mapping the latter to the former. By construction, $f$ is an isometry of $\Omega \K_{k-2}$. As such, the image of a gaussian vector by $f$ is also a gaussian vector. Hence, expressing $h_{1,k} v_1 + \cdots + h_{k-2,k} v_{k-2}$ in the new basis $w_1, \cdots, w_{k-2}$ :
\begin{align*}
    h_{1,k} v_1 + \cdots + h_{k-2,k} v_{k-2} = \eta_1 w_1 + \cdots + \eta_{k-2} w_{k-2},
\end{align*} we get random gaussian coordinates $(\eta_i)_{1 \leq i \leq k-2}$, each following the same distribution $\mathcal{N}(\mu, \sigma^2)$. We stress out that $\mathrm{Span} \langle w_1 \cdots w_s \rangle = \mathrm{Range}\left(U_{2r}\right)$. Let us denote $W_s$ the matrix which columns are formed by vectors $w_1, \cdots w_s$. Now, it is clear that
\begin{align*}
    \argmin_{x \in \R^{2r}} \| \Omega ( \mathcal{P}_{k-2}^\Omega Av_k - U_{2r} x) \|^2 = \argmin_{y \in \R^s} \| \Omega ( \mathcal{P}_{k-2}^\Omega A v_k - W_s y ) \|^2 = (\Omega W_s)^t \mathcal{P}_{k-2}^\Omega A v_k.
\end{align*}
Finally,
\begin{align*}
    \mathcal{P}_{k-2}^\Omega A v_k - W_s (\Omega W_s)^t \mathcal{P}_{k-2}^\Omega A v_k = \eta_{s+1} w_{s+1} + \cdots + \eta_{k-2} w_{k-2} =: g.
\end{align*}
The latter is a random gaussian vector of dimension $k-2 - (s+1) + 1 = k-2-s$. Recalling that $s \leq 2r$, its dimension is at least $k - 2 - 2r$. Recalling that $w_1, \cdots w_{k-2}$ are sketched orthonormal, the sketched norm of the whole vector is indeed the $\ell_2$-norm of $( \eta_i)_{s+1 \leq i \leq k-2}$.

The argument is almost identical for the case $\mathrm{dim} \; \mathrm{Range} (U_{2r}) \cap \mathrm{Span}\langle v_{k-1}, v_k \rangle = 1$. In that case, start the new basis with $w_1$ a sketched unit vector spanning this intersection, can complete it first in a sketched orthonormal basis of $\mathrm{Range}(U_{2r})$, then in a sketched orthonormal basis of $\K_{k-2}$. If the dimension is $2$, start with $w_1, w_2$ sketched orthonormal vectors spanning this intersection, then complete for $\mathrm{Range}(U_{2r})$. Clearly, they are worst case than that which we detailed, which concludes the proof.
\end{proof}

This lemma shows that if the noise of the Hessenberg matrix from~\Cref{algo:randarnoldi} is indeed gaussian, $u := v_k$ is not the vector of $\K_k \setminus \K_{k-1}$ that should be chosen to build $v_{k+1}$ by sketch orthonormalizing $A u$ with respect to $v_1, \cdots v_k$. One should choose this vector $u$ in a low-dimensional subspace of $\K_k$ and chose it such that the sketched projection of $A u$ onto $\K_k$ also lies in a low dimensional subspace of $\K_k$, and ensure access to bases of these low-dimensional subspace at each iteration.

\subsection{Approximated through randomization CG algorithm}\label{section:arCG}
The previous section shows that the traditional derivation of CG algorithm (based on symmetric tri-diagonal Hessenberg matrix) cannot be trivially extended through randomization. However we introduce in~\Cref{algo:arCG} a process that replaces the inner products of CG by sketched ones.

\begin{algorithm}
\caption{arCG: approximated through randomization CG}
\label{algo:arCG}
\KwInput{$A \in \R^{n \times n}$ an SPD matrix, \, $b \in \mathbb{R}^n$, \, $\eta \in \mathbb{R}^{+*}$, \, $k_{\text{max}} \in \mathbb{N}^*$, $x_0 \in \mathbb{R}^n$, $\Omega$ an $\epsilon$-embedding of $\K_{\kmax}$ }
\KwOutput{$x_k$ an estimate of the solution of $Ax=b$ }
\SetKwComment{Comment}{/* }{ */}
$r_0 \gets b - A x_0$ \\
$p_0 \gets r_0 $ \\
Compute $\Omega r_0$ \\
\While{$\| \Omega r_k\| \geq \eta \, \|b\|$ \textbf{and} $k \leq k_{\text{max}}$} {
\If{$k \geq 1$}{
$\delta_k \gets \frac{\|\Omega r_k\|^2}{\|\Omega r_{k-1}\|^2}$ \\
$p_k \gets r_k + \delta p_{k-1}$ \\
}
Compute $A p_k, \Omega A p_k, \Omega p_k$ \\ 
$\gamma_k \gets - \frac{\| \Omega r_k\|^2}{ \langle   \Omega A p_k,  \Omega p_k  \rangle }$ \\
$x_{k+1} \gets x_k + \gamma_k p_k$ \\
$r_{k+1} \gets r_k - \gamma_k A p_k$ \\
Compute $\Omega r_{k+1}$ \\
$k \gets k+1$
}
Return $x_k$
\end{algorithm}
We interpret~\Cref{algo:arCG} as an approximation of CG, rather than a proper randomization of CG, which we feel should fulfil the sketch Galerkin condition at each iteration. This algorithm does not. However, as we show in our numerical experiments, it can converge as fast as deterministic CG if the system is not too difficult. It can also save one synchronization per iteration if implemented in parallel. As this algorithm only imposes sketched orthogonality between last search directions and residuals, we analyzed the error relying only on this property. We were inspired by~\cite{whycg}, in which authors used the local  orthogonality approach to estimate the errors made by standard CG in finite precision arithmetics. Suppose infinite precision arithmetics for~\Cref{algo:arCG} and suppose that, for all $k \in \{1, \cdots, \kmax \}$,
\begin{align}\label{gaussquotient}
    \begin{cases}
    \gamma_k = \left( 1 \pm \tilde{\epsilon} \right) \frac{\|r_k\|^2}{\langle p_k, A p_k \rangle} \\
    \delta_k = \left( 1 \pm \tilde{\epsilon} \right) \frac{\|r_k \|^2}{\|r_{k-1}\|^2}
    \end{cases}
\end{align}
for some $\tilde{\epsilon} > 0 $. We then get, for all $k \in \{1, \cdots, \kmax - 1\}$,
\begin{align*}
    | \|x - x_k\|_A^2 - \|x - x_{k+1} \|_A^2 | & = | \| x - x_{k+1} + x_{k+1} - x_{k} \|_A^2 - \|x - x_{k+1} \|_A^2 | \\
    & \leq \|x_{k+1} - x_k \|_A^2 + 2 | \langle r_k, x_{k+1} - x_k \rangle | \\
    & \leq |\gamma_k| \left[ \frac{1 + \tilde{\epsilon}}{1 - \epsilon} \| \Omega r_k \|^2 \, + \, \frac{2 \epsilon}{1 - \epsilon} \| \Omega r_k \| \, \| \Omega p_k \| \right] \quad \text{using~\eqref{gaussquotient}}\\
    & \leq \frac{1 + \tilde{\epsilon} + 2 \epsilon}{1 - \epsilon} |\gamma_k| \| \Omega p_k \|^2 \quad \text{using } \; \| \Omega p_k \| \geq \| \Omega r_k \|.
\end{align*}
If we sum this inequation for consecutive $k$'s, we get the following result.
\begin{proposition}\label{prop:cgbound}
    Let $A \in \R^{n \times n}$ be an SPD matrix. Let $\kmax \in \mathbb{N}^*$ and let $\Omega \in \R^{\ell \times n}$ be an $\epsilon$-embedding of $\K_{\kmax}$. Assume that the coefficients $\gamma_k$ and $\delta_k$ computed by~\Cref{algo:arCG} verify~\eqref{gaussquotient} for all $k \leq \kmax$. Let $k,d \in \mathbb{N}$ such that $k+d \leq \kmax$. Then in infinite precision, estimates of~\Cref{algo:arCG} verify
\begin{align}\label{cgerror}
    \|x - x_k \|_A^2 - \|x - x_{k+d} \|_A^2 \leq \frac{1 + \tilde{\epsilon} + 2 \epsilon}{1 - \epsilon} \sum_{j = k}^{k+d-1} | \gamma_j | \| \Omega p_j \|^2.
\end{align}
\end{proposition}

\section{Experimental results} \label{section:results}
In this section we illustrate the convergence of~\Cref{algo:randarnoldi} and~\Cref{algo:arCG}.  \Cref{table} brings together various information about the matrices used in our experiments. The first Cond. column designates the original condition number of the system. The second Cond. column designates the condition number of the system after precondtionning if it was considered. \smallskip \\
\begin{table}[h]
\begin{tabular}{ |p{1.8cm}||p{1.2cm}|p{1cm}|p{2.1cm}|p{1.9cm}|p{2cm}|p{5.4cm}|}
\hline
 Name & Size & Cond. & Precond. & Cond. & Origin & Infos \\
 \hline
 Dubcova3   & 146689 & 3986 & - & -  & SuiteSparse  & -  \\
 Diff2D & 111919 & $10^{34}$ & B-Jacobi 50 & $10^5$ & FEM & 2D diffusion  \\
 El3D & 18883 & $10^{28}$ & B-Jacobi 8 & $10^6$ & FEM & Near incomp. 3D elasticity \\
 G-c5-s25 & 100000 & $10^5$ & - & - & Generated & 5 clust. radiuses $\frac{1}{4} \times$ centers\\
  G-c5-s025 & 100000 & $10^5$ & - & - & Generated & 5 clust. radiuses $\frac{1}{40} \times$ centers\\
 G-exp2 & 100000 & $10^2$ & - & - & Generated & Smooth spectrum decay \\
 G-exp3 & 100000 & $10^3$ & - & - & Generated & Smooth spectrum decay \\
 G-clust2 & 100000 & $10^2$ & - & - & Generated & 2 clust. radiuses $\frac{1}{4} \times$ centers \\
 G-clust3 & 100000 & $10^3$ & - & - & Generated & 2 clusters, size $\frac{1}{4} \times$ centers  \\
 \hline
\end{tabular}
 \caption{Set of matrices used in experiments}\label{table}
 \end{table}

\subsection{Randomized Orthogonal Projection methods}
In this section we discuss the convergence of randomized orthogonal projection methods by considering three problems arising from the discretization of different PDEs and two generated problems, each group of increasing difficulty. The matrices are described in~\Cref{table} and the convergence results are presented in \Cref{fig:RFOM}. In each problem, the sampling size of the sketching matrix $\Omega$ is set to be approximately five times the number of iterations that the deterministic method requires to converge. If a deterministic method converges in $k$ iterations, a good estimate $y_k$ to $x$ lies in the $k$-dimensional space $\K_k$. As such, we wish to embed a $k$-dimensional vector subspace with high probability. It is common to then choose $5k$ as a sampling size, as in~\cite{rgs}. 

We recall that in infinite precision arithmetic, an orthogonal projection method (OPM) produces a strictly decreasing sequence of energy norms (A-norms of the error). We chose~\Cref{algo:detarnoldi} to generate the sequence of estimates produced by OPM, since it is more stable than CG. In addition to the results presented here, ROPM was also tested on very well conditioned matrices (having condition number smaller than $10$), for which the sequence of errors produced by ROPM was also strictly decreasing and almost indistinguishable from that of OPM.

First problem, \textit{Dubcova3} was obtained from\cite{suitesparse}. Its condition number is approximately $3986$. The results are displayed in \Cref{fig:RFOM_dubcova3}.
We see already two difference with very well conditioned problems. First, the sequence of A-norm errors produced by ROPM, while still having an overall decreasing behavior, is not monotonically decreasing anymore. Small magnitude spikes of the error are observed, especially after the first iterations. Second, the overall convergence is delayed by a few iterations.
This delay becomes more important for some of the following problems. Finally, the bound from~\Cref{prop:quasi2} is indistinguishable from the ROPM error, though it is difficult to differentiate the roles of $\alpha_k$ and $\beta_k$ on this example.
Second problem is \textit{Diff2D} matrix arising from solving a 2D diffusion problem. Its initial condition number is of the order of $10^{34}$. Through block-Jacobi preconditioning with $50$ blocks, this condition number is reduced to the order of $10^5$. The results are displayed in \Cref{fig:RFOM_diff2d}. The general behavior is very close to the previous problem. However, the spikes become more frequent and their magnitude increases slightly towards the end of the iterations, when an error of $10^{-8}$ is reached. Again, the bound from~\Cref{prop:quasi2} is indistinguishable from the ROPM error, and it is hard to differentiate the roles of $\alpha_k$ and $\beta_k$.
Third problem is \textit{El3D} matrix coming from a FEM solving of a 3D  nearly incompressible elasticity problem. A more detailed description of this problem can be found in~\cite{hussam}. The initial condition number is on the order of $10^{28}$. By using block-Jacobi preconditioning with $8$ blocks, the condition number is reduced to the order of $10^6$. The convergence displayed in \Cref{fig:RFOM_el3d} highlights the previous observations. Coefficients $\alpha_k$ and $\beta_k$ are multiplied by $10^4$ simply to shift them away from the error curves on the logarithmic scale. The general behavior of the errors produced by ROPM is the same as that produced by OPM, but the magnitude of the spikes can become several orders of magnitude larger (here 3 on the largest spike). We slightly diminished the sampling size to stress out the spikes. In this numerical experiment, it seems that $\alpha_k$ and $\beta_k$ are correlated to different features of the solver's behavior. Indeed, $\alpha_k$ accurately captures the spikes; in the meanwhile, $\beta_k$ seems to be capturing the convergence delay in the last iterations. The bound is again indistinguishable from the ROPM error. The convergence rate is overall the same as that of OPM.

To further investigate these spikes, and eventually derive~\Cref{prop:quasi2}, we generated several matrices that lead to error spikes in more specific phases of the convergence. A symmetric matrix $A$ is generated by using a random orthogonal matrix and a diagonal matrix whose spectrum is purposely chosen. The random orthogonal matrix is generated as a block-diagonal matrix whose blocks are a random number of $2 \times 2$ rotation and random signs. Then $A$ is applied to a vector by applying the transpose of this orthogonal matrix, followed by the chosen diagonal matrix, followed by the orthogonal matrix. All generated matrices are of dimension $100000 \times 100000$. We choose spectra with five clusters  whose centers are exponentially decreasing. We discuss two such matrices while varying the gap between clusters as well as the gap between the eigenvalues within a cluster.

For the first matrix, the highest cluster is centered at $10^5$ and the smallest cluster is centered at $1$. The radius of the clusters is $\frac{1}{4}$ times the value of the center. The convergence results are presented in~\Cref{fig:RFOM_clust}. While we see a stairway-shaped decrease of the ROPM error, spikes of major magnitude occur at the transition between plateaus. The coefficients $\alpha_k$ accurately capture the spikes, and the bound is indistinguishable from the ROPM error. Second experiment considers the previous matrix and differs only by the size of the sketch, which is now $25$ times larger than the dimension $k$ of the subspace to be embedded. We observe in the results displayed in \Cref{fig:RFOM_clustmoresample} that the magnitude of the error spikes occurring during the plateaus are much smaller and sometimes disappear. However the error spikes occurring at the transition between two plateaus are still of important magnitude. For the last experiment we consider a  second generated problem and the sampling size is again $5$ times the dimension of the subspace we wish to embed. However, the clusters are tightened, as their radius is now $\frac{1}{40}$ times the value of the center. The results displayed in \Cref{fig:RFOM_tighterclust} have similarities with those obtained when increasing the sampling size. Similar behavior is observed as we decrease further the radius of the clusters.

In summary, these numerical experiments suggest that the spikes are related to properties of the system solved, especially its spectral properties. Indeed we observe for the matrices in our test set that clusters of larger radius lead to more error spikes inn more precise spots. Increasing the size of the sketch leads to a reduction of the magnitude of the error spikes during plateaus, but this effect is more limited for the error spikes occurring at the transition between the plateaus. The bound presented in~\Cref{prop:quasi2} was very tight in all our experiments. \Cref{fig:RFOM_el3d} suggests that $\beta_k$ is correlated to the quasi-optimality constant becoming durably larger, while $\alpha_k$ is accurately capturing the spikes. We observe a correlation between the apparition of the spikes and the most abrupt decrease of standard OPM error. In figure~\Cref{fig:RFOM_derivatives}, we show the $100$ first step of convergence of ROPM and OPM on \textit{G-clust5-s25}, with a sampling size $\ell = 1300$, and we plot the absolute value of the second (discrete) derivative of OPM's error. We see that in this example, some of ROPMs spikes coincide with spikes of this derivative. Finally, we also observe that in all our experiments, ROPM converges in almost the same number of iterations as OPM.

\begin{figure}
\begin{subfigure}[t]{.33\textwidth}
    \centering \includegraphics[width=\linewidth]{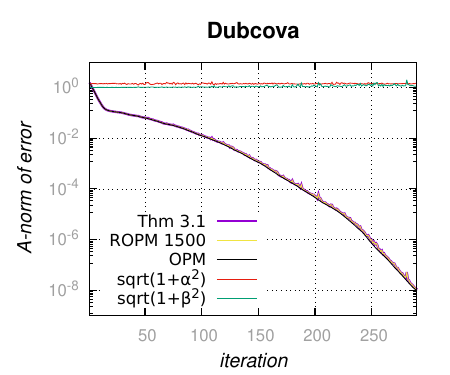}
    \caption{First natural example}
    \label{fig:RFOM_dubcova3}
\end{subfigure}
\begin{subfigure}[t]{.33\textwidth}
    \centering \includegraphics[width=\linewidth]{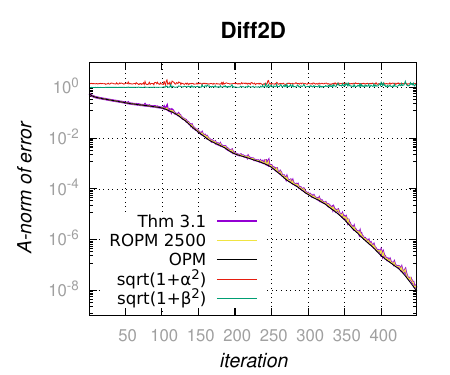}
    \caption{Harder natural example}
    \label{fig:RFOM_diff2d}
\end{subfigure}
\begin{subfigure}[t]{.33\textwidth}
    \centering \includegraphics[width=\linewidth]{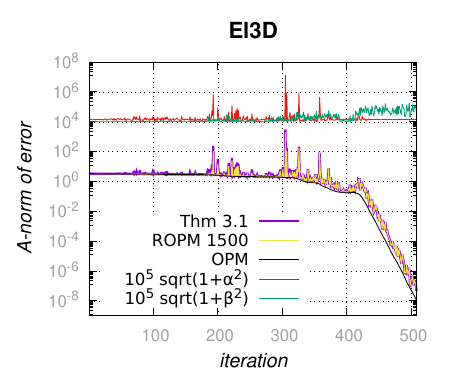}
    \caption{Harder natural example}
    \label{fig:RFOM_el3d}
\end{subfigure}
\medskip
\begin{subfigure}[t]{.33\textwidth}
    \centering \includegraphics[width=\linewidth]{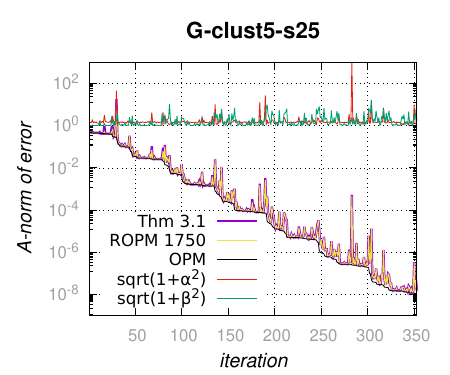}
    \caption{First generated matrix}
    \label{fig:RFOM_clust}
\end{subfigure}
\begin{subfigure}[t]{.33\textwidth}
    \centering \includegraphics[width=\linewidth]{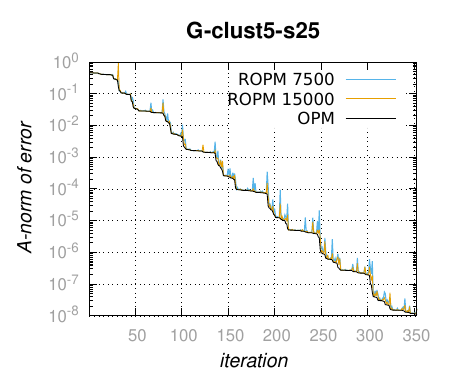}
    \caption{Greater sampling size}  \label{fig:RFOM_clustmoresample}
\end{subfigure}
\begin{subfigure}[t]{.33\textwidth}
    \centering \includegraphics[width=\linewidth]{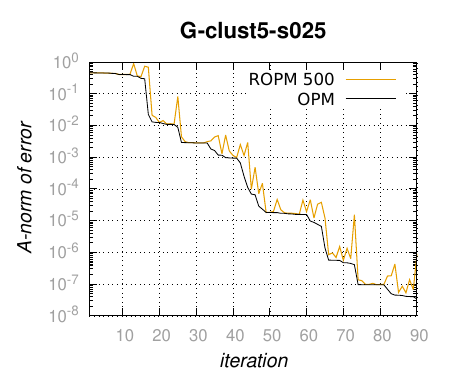}
    \caption{Tighter clusters}
    \label{fig:RFOM_tighterclust}
\end{subfigure}
\caption{Convergence of ROPM}
\label{fig:RFOM}
\end{figure}

It is worth noting that in our experiments, the Ritz values produced by deterministic and randomized algorithms (i.e eigenvalues of $\breve{H}_k$ and $H_k$) differ most on iterations corresponding to the spikes of largest magnitude. On such iterations, most of the spectra of both matrices overlap, but they differ at the extremity corresponding to the smallest Ritz values, as shown in~\Cref{fig:RFOM_spikes}.  We show the spectra of $H_{30}$ and $\breve{H}_{30}$ corresponding to the matrix \textit{G-clust5-s25} in~\Cref{fig:RFOM_it30}, where the sampling size is $\ell = 1300$. We see that the last Ritz value of $H_{30}$ is much smaller than that of $\breve{H}_{30}$. The next iteration on which the error of ROPM is close to that of OPM is the $32$nd iteration. We show the spectra of $H_{32}$ and $\breve{H}_{32}$ in~\Cref{fig:RFOM_it32}. We see that the last eigenvalue of $H_{32}$ is of the same order of magnitude as that of $\breve{H}_{32}$. In all our experiments, considering the spikes of largest magnitude, we were able to identify outliers in the spectrum of $H_k$ with respect to that of $\breve{H}_k$, and these outliers were always small eigenvalues of $H_k$. Their values can be negative, or even complex (outliers coming as a pair). If the smallest Ritz values are indeed related to the spikes, then deflating the associated vectors could participate in improving the convergence of ROPM. We leave this for future research. The authors of~\cite[Section 5.1]{guttel} also note a less regular convergence for sketched FOM compared to sketched GMRES, and relate this to the Ritz values being close to quadrature nodes used to compute an integral arising in their sketched FOM approximation. However the matrices in our test set, including the generated matrices, allow to showcase spikes of large magnitude.

\begin{figure}
\begin{subfigure}[t]{.33\textwidth}
    \centering \includegraphics[width=\linewidth]{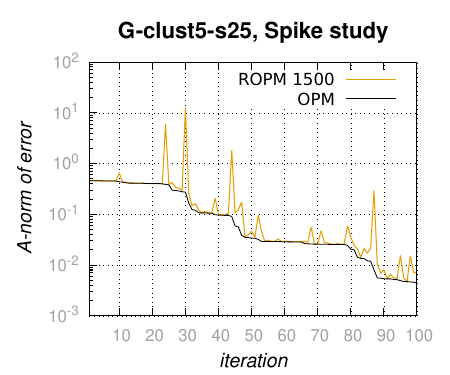}
    \caption{Large spike at iteration 30}
    \label{fig:RFOM_spikezoom}
\end{subfigure}
\begin{subfigure}[t]{.33\textwidth}
    \centering \includegraphics[width=\linewidth]{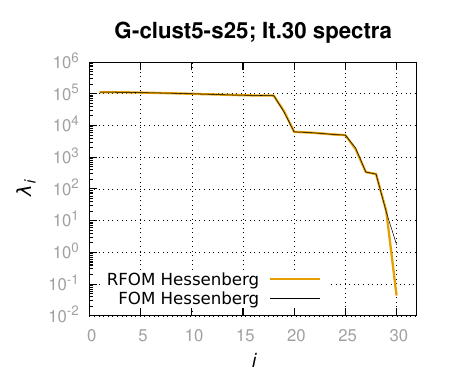}
    \caption{Spectra at spike's iteration 30}  \label{fig:RFOM_it30}
\end{subfigure}
\begin{subfigure}[t]{.33\textwidth}
    \centering \includegraphics[width=\linewidth]{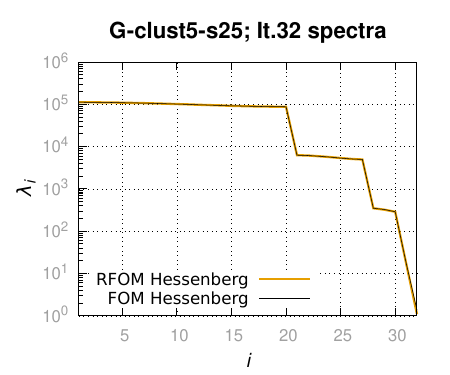}
    \caption{Spectra after spike iteration 32}
    \label{fig:RFOM_it32}
\end{subfigure}
\medskip
\hspace{5.95cm} \begin{subfigure}[t]{.33\textwidth}
    \centering \includegraphics[width=\linewidth]{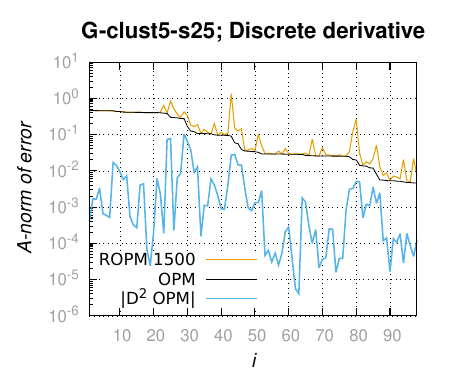}
    \caption{Discrete second derivative of OPM error ($D^2$OPM)}
    \label{fig:RFOM_derivatives}
\end{subfigure}
\caption{Spike study}
\label{fig:RFOM_spikes}
\end{figure}

\subsection{Approximation of CG}
We now discuss the convergence of approximated through randomization CG algorithm, arCG, presented in~\Cref{algo:arCG}.
We consider first two matrices arising from the discretization of PDEs.
The first matrix arises from shallow-water modelling and the convergence of arCG is displayed in \Cref{fig:arcg_shallow}. We observe that for this matrix, whose condition number is very small, $3.3$, arCG converge as fast as CG, with an A-norm error that reaches monotonically $10^{-12}$.
The second matrix arises from temperature modeling of a steel cylinder. Its condition number is $64$, and the convergence of arCG remains comparable to that of CG.

To further investigate the conditions under which arCG could be adequate, we generate four matrices with different spectral profiles.
The convergence results are presented in~\Cref{fig:arcg_exp2} to \Cref{fig:arcg_clust3}.  Matrix \textit{G-exp2}'s spectrum is set such that the eigenvalues decrease smoothly in an exponential manner as $\forall \; \; 1 \leq i \leq n, \; \; \lambda_i = \lambda_{\text{max}} \cdot \left({\lambda_{\text{min}}}/{\lambda_{\text{max}}}\right)^{\frac{i}{n-1}}$. 
Both matrices \textit{G-exp2} and \textit{G-exp3} have their minimum eigenvalue set at $\lambda_{\text{min}} = 1$. For the first matrix \textit{G-exp2}, we set $\lambda_{\text{max}} = 10^2$. For the second matrix \textit{G-exp3}, we set $\lambda_{\text{max}} = 10^3$. We see that the rate of convergence of~\Cref{algo:arCG} is comparable to that of CG for both matrices,~\Cref{fig:arcg_exp2,fig:arcg_exp3}, even if their condition number is different by one order of magnitude. For both examples, the sampling size is set to be $10$ times the number of iterations that CG took to converge. Setting the maximum eigenvalue to $10^4$ or $10^5$ gave similar results. Similar results are obtained when the eigenvalues decay linearly.

However, it is possible to slow down arCG considerably by considering instead a clustered spectrum, and then increase the condition number. We choose a second spectral profile, formed by two clusters of equal size. The gap between the two clusters is proportional to the condition number of the matrix. We set the radius of the clusters as $0.25$ times the center. Both matrices \textit{G-clust2} and \textit{G-clust3} have their cluster of small eigenvalues centered at $1$. Matrix \textit{G-clust2} has its cluster of largest eigenvalues set at $10^2$, and \textit{G-clust3} at $10^3$. We see now that increasing the condition number of the matrix slows down the rate of convergence of arCG. In the last plot, we also see the error loose its monotony. On \textit{G-clust2}, the sampling size is set to be $10$ times the number of iterations required by CG to converge, and arCG took more than $90$ iterations to converge, while CG took $37$ iterations. For matrix \textit{G-clust3}, this sampling size is not sufficient, and the error diverges. We then set the sampling size at $100$ times the number of iterations required by CG to converge. Now arCG converges again, but the error is not monotonous anymore, and arCG takes more than $640$ iterations to converge.

In summary, the condition number alone of the system does not predict the convergence of arCG when compared to CG. We experimentally identified two spectrum profiles (exponential and linear decay) for which arCG's convergence was comparable to that of CG for condition numbers up to $10^5$. We identified a spectrum profile (clustered) for which arCG's performance was much worse than that of CG as the condition number increases.

\begin{figure}[H]
  \begin{subfigure}[t]{.33\textwidth}
    \centering
    \includegraphics[width=\linewidth]{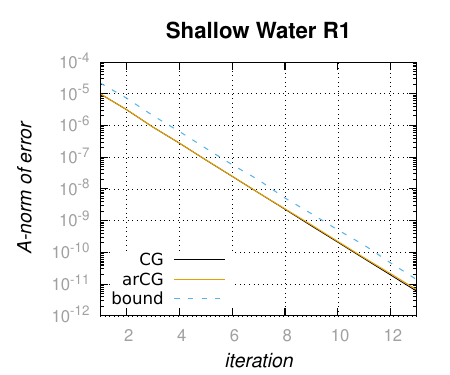}
    \caption{Typical application}\label{fig:arcg_shallow}
  \end{subfigure}
  \begin{subfigure}[t]{.33\textwidth}
    \centering
    \includegraphics[width=\linewidth]{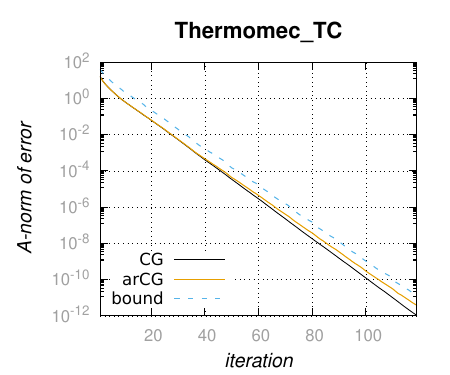}
    \caption{Harder problem}\label{fig:arcg_thermomec}
  \end{subfigure}
  \begin{subfigure}[t]{.33\textwidth}
    \centering
    \includegraphics[width=\linewidth]{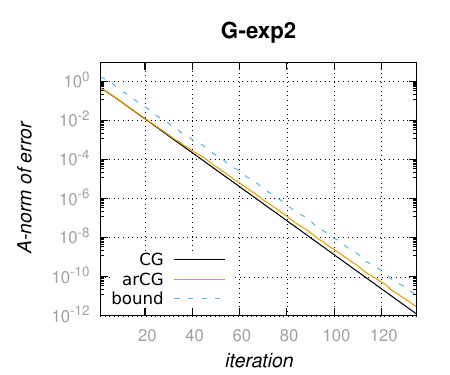}
    \caption{Well-conditioned, smooth spectral decay}\label{fig:arcg_exp2}
  \end{subfigure}
  \medskip
  \begin{subfigure}[t]{.33\textwidth}
    \centering
    \includegraphics[width=\linewidth]{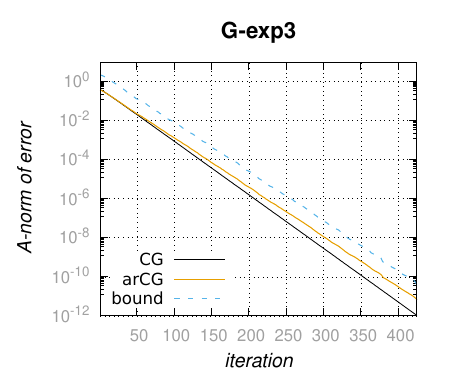}
    \caption{Ill-conditioned, smooth spectral decay}\label{fig:arcg_exp3}
  \end{subfigure}
  \begin{subfigure}[t]{.33\textwidth}
    \centering
    \includegraphics[width=\linewidth]{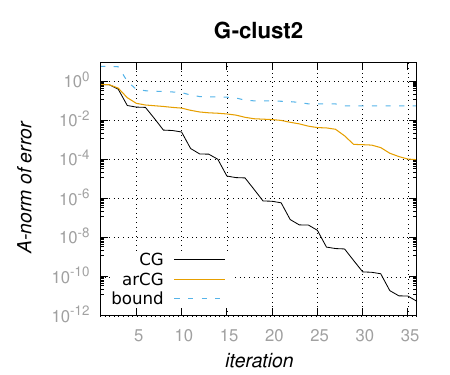}
    \caption{Well-conditioned, clustered spectral decay}\label{fig:arcg_clust2}
  \end{subfigure}
  \begin{subfigure}[t]{.33\textwidth}
    \centering
    \includegraphics[width=\linewidth]{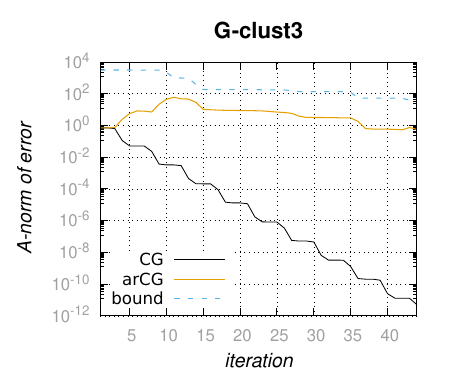}
    \caption{Ill-conditioned, clustered spectral decay}\label{fig:arcg_clust3}
  \end{subfigure}
  \caption{Convergence of arCG}
\end{figure}\label{fig:arCG}

\section{Acknowledgements}
This project has received funding from the European Research
Council (ERC) under the European Union’s Horizon 2020 research and innovation program (grant agreement No 810367).

\printbibliography

\end{document}